\documentclass[12pt]{amsart}

\usepackage[utf8]{inputenc}
\usepackage{amsmath, amsfonts, amssymb}
\usepackage{enumerate}
\usepackage{fullpage}
\usepackage{tikz-cd}
\usetikzlibrary{positioning,arrows}
\usetikzlibrary{decorations.markings}
\tikzset{%
  symbol/.style={
    draw=none,
    every to/.append style={
      edge node={node [sloped, allow upside down, auto=false]{$#1$}}
    },
  },
}

\usepackage[right]{lineno}

\usepackage{xcolor}

\def\black{\color{black}}

\numberwithin{subsection}{section}
\numberwithin{equation}{section}
\newcommand{\numberby}{equation}

\newtheorem{theorem}[\numberby]{Theorem}
\newtheorem{lemma}[\numberby]{Lemma}
\newtheorem{proposition}[\numberby]{Proposition}
\newtheorem{corollary}[\numberby]{Corollary}

\newenvironment{dremark}[1]{\refstepcounter{\numberby}%
\par \vskip 6pt  \noindent {\bf #1\ \thelemma .}}{\vskip 6pt \par}

\newenvironment{dremark*}[1]{\par \vskip 5pt \noindent 
{\bf #1.}}{\vskip 5pt \par}

\newcommand{\FLEX}{\relax}
\newcommand{\flex}[1]{\renewcommand{\FLEX}{#1}}
\newtheorem{flexthm}[\numberby]{\FLEX}
\newenvironment{flexstate}[2]{\flex{#1}\begin{flexthm}[#2]}{\end{flexthm}}

\theoremstyle{definition}
\newtheorem{remark}[\numberby]{Remark}
\newtheorem{definition}[\numberby]{Definition}
\newtheorem{example}[\numberby]{Example}
\newtheorem{notation}[\numberby]{Notation}

\newcommand{\A}{A}
\newcommand{\B}{B}
\newcommand{\C}{C}
 \newcommand{\D}{D}

\renewcommand{\H}{{\mathcal{H}}}

\newcommand{\bh}{\mathcal B(\H)}

\newcommand{\cstar}{\hbox{$C^*$}}
\newcommand{\ca}{\cstar}
\newcommand{\caop}{\cstar}
\newcommand{\cstaralg}{\cstar-algebra}
\newcommand{\cstarred}{C^*_r}

\newcommand{\iip}{ideal intersection property}
\newcommand{\riip}{regular ideal intersection property}
\newcommand{\inv}{(\textsc{inv})}
\newcommand{\inverse}{^{-1}}
\newcommand{\idealin}{\unlhd}
\newcommand{\dperp}{{\perp\perp}}

\DeclareMathOperator{\spn}{span}
\DeclareMathOperator{\reginv}{RegInv}
\DeclareMathOperator{\reg}{Reg}

\newcommand{\3}{\{1,2,3\}}
\newcommand{\aug}[1]{#1^\blacktriangle}

\newcommand{\dstext}[1]{\quad\text{#1}\quad}
\newcommand{\eps}{\ensuremath{\varepsilon}}

\providecommand{\idealin}{\unlhd}

\newcommand{\norm}[1]{\left\|{#1}\right\|}

\providecommand{\qed}%
{\hfill \vrule height5pt width4pt depth1pt \vspace{+2.00ex}}

\newcommand{\supp}{\operatorname{supp}}

\newcommand{\bbC}{{\mathbb{C}}}

\newcommand{\bbE}{{\mathbb{E}}}

\newcommand{\bbN}{{\mathbb{N}}}

\newcommand{\bbT}{{\mathbb{T}}}

\newcommand{\bbZ}{{\mathbb{Z}}}

\renewcommand{\H}{{\mathcal{H}}}

  \newcommand{\N}{{\mathcal{N}}}

  \newcommand{\U}{{\mathcal{U}}}

\newcommand{\fF}{{\mathfrak{F}}}

\newcommand{\fix}[1]{\operatorname{fix} #1}
\newcommand{\intfix}[1]{(\operatorname{fix} #1)^\circ}

\newcommand{\Iso}{\operatorname{Iso}}

\newcommand{\coexp}{\bbE}
\newcommand{\emb}{w}

\renewcommand{\theenumi}{\alph{enumi}}

\title{Regular ideals, ideal intersections, and quotients II}
\author[J.H. Brown]{Jonathan H. Brown}
\address[J.H. Brown]{
Department of Mathematics\\
University of Dayton\\
Dayton\\
OH 45469-2316\  U.S.A.} \email{jonathan.henry.brown@gmail.com}

\author[A.H. Fuller]{Adam H. Fuller}
\address[A.H. Fuller]{
Department of Mathematics\\
Ohio University\\
Athens\\
OH 45701\  U.S.A.}
\email{fullera@ohio.edu}

\author[D.R. Pitts]{David R. Pitts}
\address[D.R. Pitts]{
  Department of Mathematics\\
  University of Nebraska-Lincoln\\
  Lincoln\\
  NE 68588\ 
  U.S.A.}  \email{dpitts2@unl.edu}

\author[S.A. Reznikoff]{Sarah A. Reznikoff}
\address[S.A. Reznikoff]{
Department of Mathematics\\
Virginia Tech\\
Blacksburg, VA 24061-1026\  U.S.A. }
\email{reznikoff@vt.edu}

\subjclass[2020]{46L05}

\begin{document}

\begin{abstract}
  Let $B \subseteq A$ be a regular inclusion of
  \cstaralg s satisfying the ideal intersection property and with a
  faithful invariant pseudo-expectation.  A complete description of
  the regular ideals of $A$ is given using the invariant regular
  ideals of $B$ and the pseudo-expectation.  Further, necessary and
  sufficient conditions are given for a quotient by a regular ideal to
  preserve the faithful unique pseudo-expectation property.

Special attention is given throughout to pseudo-Cartan inclusions,
i.e.\ regular inclusions with the faithful unique pseudo-expectation
property, equivalently, having a Cartan envelope.  We show that the quotient
of a pseudo-Cartan inclusion by a regular ideal is again a
pseudo-Cartan inclusion, and we describe the Cartan envelope of the
quotient. 
\end{abstract}

\maketitle

\section{Introduction}
Let $B\subseteq A$ be an inclusion of $\ca$-algebras.  In our previous
paper,~\cite{BFPR24}, we showed that under suitable hypotheses, there
are isomorphisms between the lattices of 
regular ideals of $A$ and  regular ideals of $B$.  The conditions
imposed there involved conditional expectations that satisfy an
invariance hypothesis, and particular attention was paid to Cartan
inclusions $B\subseteq A$.

However,  inclusions $\B\subseteq \A$
for which there is no conditional expectation of $\A$ onto $\B$ arise quite
naturally, see~\cite[Example 1.2]{Pit17}.   Inspired by this example, let $X$ be a connected, compact
Hausdorff space and $\Gamma$  a discrete abelian group acting on $X$
via homeomorphisms; denote the relative commutant of $C(X)$ in the
crossed product $C(X)\rtimes\Gamma$ by $C(X)^c$. Then $C(X)^c$ is a
regular and maximal 
abelian \cstar-subalgebra of $C(X)\rtimes \Gamma$. 
If there exists $s\in \Gamma$ such that
$\emptyset \subsetneq (\fix s)^\circ\subsetneq X$, where $\fix s$ is
the set of fixed points for the homeomorphism of $X$ corresponding to
$s$,  Proposition~\ref{noCEfam} below shows there is no conditional
expectation $\Delta:C(X)\rtimes\Gamma\rightarrow C(X)^c$. Other
examples of regular inclusions $\B\subseteq \A$ having no conditional
expectation arise when considering \cstaralg s of higher rank graphs,
see Example~\ref{2grExPre} below.

In the present article, we extend some of
the results from \cite{BFPR24} to settings where conditional
expectations may not be present. 
As a replacement for a conditional expectation, we use
pseudo-expectations.  Hamana showed that every $\ca$-algebra $A$ can
be rigidly embedded into an injective $\ca$-algebra $I(A)$, an injective
envelope for $A$ \cite{Ham79}.  Let $A$ be a $\ca$-algebra containing
a $\ca$-algebra $B$ and let $\iota$ be a rigid embedding of $B$ into its
injective envelope $I(B)$.  A pseudo-expectation is a completely contractive
positive map $\Phi \colon A \rightarrow I(B)$ which extends $\iota$.
Pseudo-expectations were introduced by Pitts as generalizations of
conditional expectations \cite{Pit17}.  In recent years, injective
envelopes and pseudo-expectations have played an important role in the
study of $C^*$-algebras and their ideal structure.  Examples include
Kalantar and Kennedy's classification of $\ca$-simple groups
\cite{KalKen17}; Kennedy and Schafhauser's work on crossed products
\cite{KenSch19}; the work of Borys \cite{Borys2020} and
Kennedy-Kim-Li-Raum-Ursu \cite{KennedyGroupoid} on groupoid
$\ca$-algebras; and the work of Pitts \cite{Pit17, Pit21} and
Pitts-Zarikian \cite{PitZar15} on regular inclusions of
$\ca$-algebras.

We shall pay particular attention to pseudo-Cartan inclusions, which
are regular inclusions $\D\subseteq\A$ with $\D$ abelian such that there is
a unique pseudo-expectation $E:\A\rightarrow I(\D)$ which is also
faithful.  (In general, there are many pseudo-expectations for an
inclusion.)  The notion of pseudo-Cartan inclusion is a natural
generalization of of Cartan inclusion: every Cartan inclusion is a
pseudo-Cartan inclusion, and a pseudo-Cartan inclusion is a Cartan
inclusion if and only if the pseudo-expectation ``is'' a conditional
expectation in the sense of
\cite[Proposition~2.3.19]{Pit25}. Furthermore, a regular inclusion
$\D\subseteq\A$ with $\D$ abelian is a pseudo-Cartan inclusion if and only
if $\D\subseteq\A$ can be minimally and regularly embedded into a Cartan
inclusion, i.e.\ has a Cartan envelope, see
~\cite[Theorem~3.27]{Pit25}.  The examples mentioned in the second
paragraph are pseudo-Cartan inclusions.  Also,
\cite[Proposition~2.4.4]{Pit25} shows that an inclusion
$\D\subseteq \A$ with $\A$ is abelian is a pseudo-Cartan inclusion if
and only if $\D\subseteq \A$ has the \iip; in particular, the example
found in~\cite[Section~5]{Exel23} is a pseudo-Cartan inclusion.

Regular ideals were introduced by Hamana in his work on monotone
completions of $\ca$-algebras.  Restricting the
one-to-one correspondence between the open sets of the primitive ideal
space $\mathrm{Prim}(A)$ and the ideals of $A$ to the complete Boolean
algebra of regular open sets of $\mathrm{Prim}(A)$ produces the
regular ideals of $A$, i.e., ideals $J \unlhd A$ satisfying
$J = J^{\perp\perp}$. Hamana further showed  there is a Boolean
isomorphism between the regular ideals of $A$ and the projection
lattice of $Z(I(A))$, where $Z(I(A))$  is the center of $I(A)$ \cite{Ham81, Ham82}. 

This
connection between regular ideals and an injective envelope was
exploited by Pitts and Zarikian to show that certain nice properties
are preserved when taking quotients by regular ideals
\cite{PitZar15}.  Explicitly, Theorem 3.8 of \cite{PitZar15} says that
if $D$ is an abelian $\ca$-subalgebra of a $\ca$-algebra $A$, $J$ is
an ideal of $A$ such that $D\cap J$ is a regular ideal of $D$, and
there is a unique pseudo-expectation $\Phi \colon A \rightarrow I(D)$,
then the inclusion $D/(J\cap D) \subseteq A/J$ also has a unique
pseudo-expectation.  In \cite[Example~4.3]{PitZar15}, Pitts and
Zarikian give an example where the unique pseudo-expectation is
faithful, but the unique pseudo-expectation in the quotient is not
faithful.

Theorem 3.8 of \cite{PitZar15} was the initial impetus for our study
of regular ideals (and quotients thereof) in inclusions of
$\ca$-algebras.  We began by considering regular ideals in graph
algebras in \cite{BFPR22}.  Let $\ca(E)$ be a graph $\ca$-algebra and
let $D$ be the algebra in $\ca(E)$ generated by the range projections
of the partial isometries in the Cuntz-Krieger family generating
$\ca(E)$.  In \cite{BFPR22} we give a complete description of the
gauge invariant regular ideals of $\ca(E)$ in terms of the directed
graph $E$.  Further, we showed that if $D$ is a Cartan subalgebra of
$\ca(E)$ (i.e. if $E$ is satisfies condition (L) of
\cite{KumPasRaeRen97}) and $J$ is a regular ideal of $\ca(E)$, then $\ca(E)/J$ is a graph algebra and
$D/(J\cap D)$ is a Cartan subalgebra of $\ca(E)/J$.  These results
were generalized to $k$-graph $\ca$-algebras by Schenkel \cite{Sch22}
and more general graph $\ca$-algebras by Va\v{s} \cite{Vas23};
algebraic analogues of these results are given for Leavitt-path
algebras by Gon\c{c}alves-Royer \cite{GonRoy22} and Va\v{s}
\cite{Vas23} and for Kumjian-Pask algebras by Schenkel \cite{Sch22}.

We generalized the graph algebra results of \cite{BFPR22} to arbitrary
Cartan inclusions $D \subseteq A$ in \cite{BFPR24}. We showed that if
$J$ is a regular ideal in $A$ then $D/(J\cap D)$ is a Cartan
subalgebra of $A/J$.  In particular, these results say that when $D$
is a Cartan subalgebra of $A$ then the quotient by a regular ideal of
$A$ preserves not just the unique pseudo-expectation property (as
guaranteed by \cite{PitZar15}), but the \emph{faithful} unique
pseudo-expectation property.  Both \cite{BFPR22} and \cite{BFPR24}
deal with regular ideals in inclusions with a faithful conditional
expectation, with particular emphasis on Cartan inclusions.

Many results in \cite{BFPR24} apply to settings beyond Cartan
inclusions.  For example, given the inclusion $(\A,\B)$,
\cite[Theorem~3.24]{BFPR24} shows that under suitable hypotheses,
including the existence of a faithful conditional expectation
$E: \A\rightarrow \B$, the map $J \mapsto J \cap B$ is a Boolean
isomorphism from the Boolean algebra of regular ideals of $A$ to the
invariant regular ideals of $B$; this is a key result of
\cite{BFPR24}.  Exel generalized this result in \cite{Exel23}.  A
special case of Exel's result is that the map $J \mapsto J \cap B$ can
be shown to be a Boolean isomorphism without the assumption that there
is a faithful conditional expectation (or any conditional expectation)
from $A$ to $B$.  However, the additional hypothesis in \cite{BFPR24}
of a conditional expectation has an advantage: in
\cite[Theorem~3.24]{BFPR24} we give an explicit description of the
inverse of the map $J \mapsto J \cap D$ using a faithful conditional
expectation.  This has implications for quotients of inclusions by
regular ideals (see \cite[Theorem~4.2]{BFPR24}), and in particular
leads to \cite[Theorem~4.8]{BFPR24}: if $(A,D)$ is a Cartan inclusion
and $J\idealin A$ is a regular ideal, then $(A/J, D/(D\cap J))$ is
again a Cartan inclusion.

In this article, we focus on the regular ideals in relation to
pseudo-expectations.  In Theorem~\ref{thm: 1-1 reg ideals pseudo}, we
describe the inverse of the map $J \mapsto J \cap B$ on regular ideals
$J \unlhd A$ for certain regular inclusions of $\ca$-algebras $B
\subseteq A$ satisfying the \iip\ and having a faithful
pseudo-expectation.  This result generalizes
\cite[Theorem~3.24]{BFPR24}.  We do not assume $B$ is abelian.
However, pseudo-Cartan inclusions are an interesting setting where $B$
is abelian, and the hypotheses of Theorem~\ref{thm: 1-1 reg ideals
pseudo} are satisfied for pseudo-Cartan inclusions.

Let $D \subseteq A$ be pseudo-Cartan, and let $D_1 \subseteq A_1$ be a
Cartan envelope, via the embedding $\alpha \colon A \rightarrow A_1$.
In Theorem~\ref{thm: all latice isomorphisms}, we show that there is a
Boolean isomorphism between the regular ideals of $A$ and the regular
ideals of $A_1$.  The key step is to show that the set of
invariant regular ideals of $D$ and  the set of invariant
regular ideals of $D_1$ are isomorphic.

In Section~\ref{sec: quotients}, we address quotients by regular
ideals.  Previous results in \cite{PitZar15, BFPR22, Sch22, BFPR24} explore
settings where quotients by regular ideals preserve the faithful unique
pseudo-expectation property.  In Theorem~\ref{thm: when quot has
  faithful}, we give necessary and sufficient conditions for when a 
quotient by a regular ideal preserves the faithful unique
pseudo-expectation property.  That is, we give necessary and
sufficient conditions for a version of \cite[Theorem~3.8]{PitZar15}
for \emph{faithful} unique pseudo-expectations.  We apply this result
to the two main types of inclusions of interest: regular inclusions
(Corollary~\ref{cor: when quot has faithful}) and inclusions of
abelian $\ca$-algebras (Corollary~\ref{cor: abelian quotients}).
Finally, we return to pseudo-Cartan inclusions.  A main result,
Theorem~\ref{thm: pseudo cartan quotient}, shows that the quotient of
a pseudo-Cartan
inclusion by a regular ideal is again a pseudo-Cartan inclusion, and the Cartan
envelope of the quotient is the quotient
of the Cartan envelope of the original pseudo-Cartan inclusion. This
considerably generalizes~\cite[Theorem~4.8]{BFPR24}.

In Section~\ref{Sec:RCP}, we consider reduced crossed products of $C(X)$ by a
discrete group $\Gamma$.  When $X$ is a connected and compact
Hausdorff space, we give a dynamical characterization of when
$C(X)^c\subseteq C(X)\rtimes_r\Gamma$ is a Cartan inclusion, where
$C(X)^c$ is the relative commutant of $C(X)$ in $C(X)\rtimes_r
\Gamma$, see
Theorem~\ref{connect}.  Since $C(X)^c\subseteq C(X)\rtimes_r\Gamma$
is always a pseudo-Cartan subalgebra, Theorem~\ref{connect} gives a
means for constructing a large class of pseudo-Cartan inclusions which
are not Cartan inclusions.

Finally, Section~\ref{Sec:Examples} is devoted to a number of examples
arising from reduced crossed products and higher-rank graph \cstaralg s,
which illustrate our results.

\section{Preliminaries on inclusions of $C^*$-algebras}\label{Sec:
prelim}

By an \emph{inclusion of $\ca$-algebras} we mean a $\ca$-algebra $A$
and a $\ca$-subalgebra $B \subseteq A$ such that $B$ contains an
approximate unit for $A$.  (We caution the reader that the approximate
unit hypothesis for inclusions is not always found in the literature: for example,
this is not a blanket assumption in~\cite{Pit25}.) We will denote an inclusion $B \subseteq A$
by $(A,B)$ (with the largest algebra first).  An element $n \in A$ is
a \emph{normalizer} of $B$ if $nB n^*,\ n^*Bn \subseteq B$.  Given an
inclusion $(A, B)$ we denote by $N(A,B)$ the collection of all
normalizers of $B$ in $A$.  If $N(A,B)$ span a dense subset of $A$ we
say that the inclusion $(A,B)$ is a \emph{regular inclusion of
$\ca$-algebras}.  If $N \subseteq N(A,B)$ is a $*$-semigroup such that
the span of $N$ is dense in $A$, then we call $N$ a \emph{generating
$*$-semigroup} for the regular inclusion $(A,B)$.

The assumption that $\B$ contains an approximate unit for $\A$ is
strong enough to ensure regularity of inclusions in the abelian case.

\begin{lemma}\label{abel+AUP} If $(\A,\B)$ is  an inclusion of
  \cstaralg s and 
  $\A$ is abelian, then $(\A,\B)$ is regular.
\end{lemma}
\begin{proof} 
 Let $(e_\lambda)$ be an approximate unit for $\A$ which is contained in $\B$.
If $\A$ is unital, then  $I_\A=\lim_\lambda I_\A e_\lambda
=\lim_\lambda e_\lambda$, whence $I_\A\in \B$.  On the other
  hand, if $\B$ is unital, then 
  $e_\lambda\rightarrow I_\B$, and it follows that $I_\B$ is
  the unit for $\A$.  If either of these cases hold, both $\A$ and
  $\B$ are unital and $I_\A=I_\B$.  Since every unitary in $\A$
  normalizes $\B$ and every element of $\A$ is a linear combination of
  four unitaries, $(\A,\B)$ is regular.

  Thus we may as well assume that both $\A$ and $\B$ are non-unital.
  Let $a\in \A$.  Since $\B$ contains an approximate unit for $\A$,
  given $\eps>0$, there exists $0\leq b\in \B$ with $\norm{b}\leq 1$
  and $\norm{ab-a}<\eps$.  We may find
  $\{(x_j,\lambda_j)\}_{j=1}^4\subseteq \U(\tilde \A)$ and scalars
  $c_j$ such that
 \begin{equation}\label{abel+AUP0}
   (a,0)=\sum_{j=1}^4 c_j(x_j,\lambda_j).
 \end{equation}
As $(x_j,\lambda_j)$ is a unitary in $\tilde \A$,   $(x_j,
\lambda_j)(b,0)\in N(\tilde\A,\tilde\B)$, 
hence \[n_j:=x_jb+\lambda_j b\in N(\A,\B).\]
Thus,
 \[\norm{a-\sum_{j=1}^4
       c_jn_j}=\norm{(a,0)-\sum_{j=1}^4c_j
     (x_j,\lambda_j)(b,0)}=\norm{a-ab}<\eps,\] showing that $a\in
 \overline{\spn}\, N(\A,\B)$.  Thus $(\A,\B)$ is regular.
\end{proof}

Hamana \cite{Ham79} showed that any unital $\ca$-algebra $A$ has a
(unique) \emph{injective envelope}.  That is, there is an 
injective $\ca$-algebra $I(A)$ and a $*$-monomorphism $\iota:
\A\rightarrow I(\A)$ such that whenever $\phi: I(\A)\rightarrow I(\A)$
is a unital and completely positive map satisfying
$\phi\circ\iota=\iota$, then $\phi$ is the identity map on $I(\A)$.  When $A$ is
a non-unital $C^*$-algebra, an injective envelope for $A$ is defined
to be an injective envelope for the unitization of $A$, $\tilde{A}$.
We will denote an injective envelope for $A$ by $(I(A), \iota)$.  We
refer the reader to \cite{Ham79} and \cite[Section~2.3]{Pit25} for
more details on the definition of the injective envelope and the
precise uniqueness statement.

Let $(I(A),\iota)$ be an injective envelope for a $\ca$-algebra $A$.
If $C \subseteq A$ is a hereditary $\ca$-subalgebra, then by
\cite[Lemma~1.1]{Ham81}, there is a unique projection $p \in I(A)$ such that
$(pI(A)p, \iota|_C)$ is an injective envelope for $C$.  We call the
projection $p$ the \emph{structure projection} for $C$.  
When $C$ is an ideal of $A$, the structure projection $p$ is a central
projection in $I(A)$.

Let $(A,B)$ be an inclusion of $\ca$-algebras, and $(I(B),\iota)$ be
an injective envelope for $B$.  A
normalizer $n \in N(A,B)$ determines partial dynamics on $B$ and also
on the injective envelope $I(B)$.  We show this in Lemma~\ref{lem:
theta maps} below.  Lemma~\ref{lem: theta maps} is a direct
generalization of \cite[Lemma~2.1]{Pit17} and can be seen as a
``point-free" version of \cite[Proposition~1.6]{Kum86}.  We begin with
a preliminary lemma.

\begin{lemma}[{cf.~\cite[Proposition~II.3.4.2]{Blackbook}}]\label{lem:
hereditary subalg} Let $(A,B)$ be an inclusion of $\ca$-algebras and
let $n \in N(A,B)$ be a normalizer.  Then $\overline{nBn^*}$ is a
hereditary $\ca$-subalgebra of $B$ and $\overline{nBn^*} =
\overline{|n^*| B |n^*|}$.
\end{lemma}

\begin{proof} Note that $nn^*, n^*n \in B$ by the assumption that $B$
contains an approximate unit for $A$.  Thus $|n^*| = (nn^*)^{1/2} \in
\overline{nBn^*}$, and so $|n^*| B |n^*| \subseteq \overline{nBn^*}.$
Further, for any $b \in B$,
\begin{align*} nbn^* & = \lim_k (nn^*)^{1/k} nbn^* (nn^*)^{1/k} \\ &=
\lim_k |n^*|^{2/k} nbn^* |n^*|^{2/k}.
\end{align*} Thus $nBn^* \subseteq \overline{ |n^*| B |n^*|}$.  It
follows that $\overline{nBn^*} = \overline{ |n^*| B |n^*|}$ and is
thus a hereditary $\ca$-subalgebra of $B$, by
\cite[Corollary~3.2.4]{MurphyBook}.
\end{proof}

\begin{lemma}\label{lem: theta maps} Suppose $(A,B)$ is an inclusion of
$\ca$-algebras and  $(I(B),\iota)$ is an injective envelope for
$B$.  Let $n \in N(A,B)$ be a normalizer and let $p,q \in I(B)$ be the
structure projections for the hereditary $\ca$-subalgebras
$\overline{nBn^*}$ and $\overline{n^*B n}$ respectively.

  Then the map $|n^*|\, b\, |n^*| \mapsto n^*bn$  is well-defined and uniquely
 extends to isomorphisms
$$\theta_n \colon \overline{nBn^*} \rightarrow \overline{n^*B n}\quad\text{
  and } \quad\tilde{\theta}_n \colon pI(B)p \rightarrow qI(B)q. $$
Furthermore,
\[\tilde\theta_n\circ\iota|_{\overline{nBn^*}}=\iota\circ\theta_n.\]
\end{lemma}

\begin{proof}  Fix $n\in N(A,B)$.  We first show the map $|n^*|\, b\, |n^*|
\mapsto n^*bn$ is well-defined.  To do this, we show that for $b\in
B$,
  \begin{equation*}\label{thetam1} |n^*|\, b\, |n^*| =0 \iff n^*bn=0.
  \end{equation*} Since $n^*=\lim_{j\rightarrow \infty} |n^*n|^{1/j}
n^*$, there is a sequence of polynomials $p_k$ with real coefficients
such that $p_k(0)=0$ and $\lim_k p_k(n^*n)n^* =n^*$.  Now suppose
$b\in B$ and $|n^*| b |n^*|=0$.  Multiplying on the left by
$n^*|n^*|$ gives $(n^*n)n^*b |n^*|=0$, whence
  \[n^*b|n^*|=\lim_k p_k(n^*n)n^*b|n^*|=0.\] Likewise, multiplying
$0=n^*b |n^*|$ on the right by $|n^*| n$ yields
  \[0=\lim_k n^*b np_k(n^*n)=n^*bn.\] Similar arguments give the
converse.  It follows the map $|n^*|\, b\, |n^*| \mapsto n^*bn$ is
well-defined.

  Next, assume that $A$ sits inside an AW$^*$-algebra $W$ (e.g. by
taking a faithful representation of $A$ into some $B(H)$ or by
embedding $A$ into its injective envelope $I(A)$).  Then there is a
partial isometry $u^* \in W$ so that $n^* = u^*|n^*|$ is the polar
decomposition of $n^*$.  Thus $n = |n^*|u$, $|n^*| = nu^* = un^*$, and
$u^*u n^* = n^*$.  Take any $b \in B$. Then
\begin{align*} \| n^* b n \| & = \| u^* u n^* b n u^* u \| \\ & \leq
\| u n^* b n u^* \| = \| |n^*|\, b\,  |n^*| \| \\ & \leq \| n^* b n\|.
\end{align*} Thus the map $|n^*|\, b\, |n^*| \mapsto n^* b n$ is
isometric.  By Lemma~\ref{lem: hereditary subalg}, the existence of
the map $\theta_n$ now follows as in \cite[Lemma~2.1]{Pit17}.  The
existence of the map $\tilde{\theta}_n$ follows from
\cite[Lemma~1.1]{Ham81} as $(pI(B)p, \iota|_{\overline{nBn^*}})$ and
$(qI(B)q, \iota|_{\overline{n^*Bn}})$ are injective envelopes of
$\overline{nBn^*}$ and $\overline{n^*Bn}$, respectively.   That 
$\iota\circ \theta_n=\tilde\theta_n\circ\iota|_{\overline{nBn^*}}$ follows likewise.
\end{proof}

Let $A$ and $B$ be $\ca$-algebras with $B \subseteq A$ and let
$(I(B),\iota)$ be an injective envelope for $B$.  Following
\cite{Pit17, Pit25}, we call a completely positive contractive map
$\Phi\colon A \rightarrow I(B)$ a \emph{pseudo-expectation} if
$\Phi|_B = \iota$.  That is, a completely positive contraction $\Phi$
is a pseudo-expectation if the following diagram commutes
\[
\begin{tikzcd} A \arrow[rd, "\Phi"] & \\ B \arrow[u, symbol=\subseteq]
\arrow[hook, r, "\iota"] & I(B).
\end{tikzcd}
\] A pseudo-expectation always exists by \cite[Lemma~2.3.8]{Pit25}.
When $\Phi$ is a pseudo-expectation and $\Phi(A) =\iota(B)$, we
suppress $\iota$ and simply write $\Phi \colon A \rightarrow B$.  In
this case $\Phi$ is a \emph{conditional expectation}.

Like conditional expectations, pseudo-expectations are always bimodule
maps. 

\begin{lemma}[{\cite[Proposition~2.3.11]{Pit25}}]\label{lem:
bimod} Let $(A,B)$ be an inclusion of $\ca$-algebras and let
$(I(B),\iota)$ be an injective envelope for $B$.  Let $\Phi \colon A
\rightarrow I(B)$ be a pseudo-expectation.  Then
$$\iota(b_1)\Phi(a) \iota(b_2) = \Phi(b_1ab_2)$$
for all $b_1,\ b_2 \in B$ and $a \in A$.
\end{lemma}

An inclusion $(A,B)$ of $\ca$-algebras has the \emph{ideal
intersection property} if $J \cap B \neq \{0\}$ for every non-zero
ideal $J \unlhd A$.  The following lemma, adapted from
\cite[Proposition~3.1]{PitZar15}, relates the \iip\ to
pseudo-expectations.

\begin{proposition}\label{prop: iip and faithful pseudo} Let $(A,B)$
be an inclusion of $\ca$-algebras.  If $J \unlhd A$ is an ideal such
that $J \cap B = \{0\}$, then there is a pseudo-expectation $\Phi
\colon A \rightarrow I(B)$ such that $\Phi(x^*x) = 0$ for all $x \in
J.$

  In particular, if all pseudo-expectations $\Psi \colon A \rightarrow
I(B)$ are faithful, then $(A,B)$ has the \iip.
\end{proposition}

\begin{proof} Let $J \unlhd A$ be an ideal such that $J \cap B =
\{0\}$.  Define a map $E \colon J + B \rightarrow B$ by $E(x
+ b) = b.$ As $I(B)$ is injective, $E$ extends to a
pseudo-expectation $\Phi \colon A \rightarrow I(B)$.  By construction,
$\Phi(x^*x) = 0$ for all $x \in J$.
\end{proof}

\subsection{Pseudo-Cartan inclusions} The theory of
pseudo-expectations is better developed for inclusions $(A,D)$ where
$D$ is abelian.  Indeed, this is the setting studied in \cite{Pit17,
Pit21, PitZar15}.  Lemma~\ref{lem: theta maps} will allow us to avoid
this assumption in several places.  However, $D$ is abelian in the
main class of examples of interest to us: pseudo-Cartan inclusions
$(A,D)$.  We end this section with a summary of the key facts on
pseudo-Cartan inclusions.  The key definitions and Theorem~\ref{thm:
pseudo-Cartan equiv} are from \cite{Pit21, Pit25}; the term
\emph{pseudo-Cartan} was coined in \cite{Pit25}.

\begin{definition} A regular inclusion of $\ca$-algebras $(A,D)$ is a
\emph{pseudo-Cartan inclusion} if
    \begin{enumerate}
    \item $D$ is abelian;
    \item $D^c$, the relative commutant of $D$ in $A$, is abelian;
    \item\label{pcdef3} both the inclusions $(D^c, D)$ and $(A,D^c)$ have the \iip.
    \end{enumerate} A pseudo-Cartan inclusion $(A,D)$ is called
\emph{a virtual Cartan inclusion} if $D$ is maximal abelian in $A$.
By condition~\eqref{pcdef3}, every 
pseudo-Cartan inclusion has the \iip.

\end{definition}
By combining~\cite[Proposition~2.3.19 and
Theorem~3.27]{Pit25}, one finds that a virtual Cartan inclusion $(A,D)$ is a Cartan
inclusion if and only if there is a (necessarily faithful)
conditional expectation $E \colon A \rightarrow D$.

Pseudo-Cartan inclusions are precisely the regular inclusions that have a
Cartan envelope.  We recall the precise definition of a Cartan
envelope in Definition~\ref{def: Cartan env} below, and record several
equivalent conditions for an inclusion to be pseudo-Cartan in
Theorem~\ref{thm: pseudo-Cartan equiv}.  We first recall the
definition of a regular morphism of regular inclusions from
\cite{Pit25}.

\begin{definition}
  Let $(A,B)$ and $(A_1,B_1)$ be regular inclusions of $\ca$-algebras.
  A \emph{regular morphism} $\alpha$ from $(A,B)$ to $(A_1,B_1)$,
  denoted
 $$\alpha \colon (A,B) \rightarrow (A_1,B_1),$$
 is a homomorphism $\alpha \colon A \rightarrow A_1$ such that
 $\alpha(N(A,B)) \subseteq N(A_1,B_1)$.  If the regular morphism
 $\alpha$ is a $^*$-monomorphism and $\alpha(B)\subseteq B_1$, we call
 $\alpha$ a \emph{regular expansion}, see \cite{Pit25}.
\end{definition}

\begin{remark} Let $(A,D)$ and $(A_1,D_1)$ be regular inclusions with
$D$ and $D_1$ abelian, and let $\alpha \colon (A,D) \rightarrow
(A_1,D_1)$ be a regular $*$-monomorphism.  A consequence of our
standing assumption that inclusions have the approximate unit property
is that $\alpha$ is automatically a regular expansion (see
\cite[Lemma~3.3(b)]{Pit25}).
\end{remark}

\begin{definition}[{\cite[Definition~3.13]{Pit25}}]\label{def: Cartan
env} Let $(A,D)$ be a regular inclusion with $D$ abelian.  A
\emph{Cartan envelope} of $(A,D)$ is a triple $(A_1,D_1, \alpha)$ such
that
\begin{enumerate}
\item \label{def: Cartan env1} $(A_1,D_1)$ is a Cartan inclusion;
\item \label{def: Cartan env2} $\alpha \colon (A,D) \rightarrow (A_1,D_1)$ is a regular
expansion with $A_1 = \cstar(\alpha(A) \cup E_1(\alpha(A)))$ and $D_1 =
\cstar(E_1(\alpha(A)))$, where $E_1 \colon A_1 \rightarrow D_1$ is the
faithful conditional expectation;
\item \label{def: Cartan env3} $(D_1,\alpha(D))$ has the \iip.
\end{enumerate}
\end{definition}

\begin{definition}[{\cite{PitZar15}}] Let $(A,B)$ be a regular inclusion of
C$^*$-algebras.  We say that $(A,B)$ has the \emph{faithful unique
pseudo-expectation property} if there is unique pseudo-expectation
$\Phi \colon A \rightarrow I(B)$ and $\Phi$ is faithful.
\end{definition}

\begin{theorem}[{c.f.~\cite[Theorem~3.27]{Pit25}}] \label{thm:
pseudo-Cartan equiv} Let $(A,D)$ be a regular inclusion of
$\ca$-algebras with $D$ abelian.  The following statements are equivalent:
\begin{enumerate}
    \item $(A,D)$ is a pseudo-Cartan inclusion;
    \item $(A,D)$ has a Cartan envelope;
    \item $(A,D)$ has the faithful unique pseudo-expectation property;
    \item $(\tilde{A},\tilde{D})$ is a pseudo-Cartan inclusion, where
$\tilde{A}$ and $\tilde{D}$ are the unitizations of $A$ and $D$
respectively.
\end{enumerate}
\end{theorem}

\begin{remark} Condition (4) in Theorem~\ref{thm: pseudo-Cartan equiv}
is equivalent to the other three conditions by
\cite[Observation~4.2.1]{Pit25} since we are assuming $A$ and $D$
share an approximate unit.   Also, when $(A,D)$ has a Cartan envelope,
it is minimal and unique; see~\cite[Theorem~3.27]{Pit25} for the
precise statements.
\end{remark}

\begin{remark}\label{nexttryA.0}   See~\cite[Section 2]{PittsCoStReInII} for an example
  of a pseudo-Cartan inclusion whose Cartan envelope is explicitly
  described.  Additional examples are found in
  Section~\ref{Sec:Examples}.
  \end{remark}

\begin{example}[Higher-rank graph algebras]\label{ex: k-graph} The universal algebra  $\cstar(\Lambda)$ of a $k$-graph $\Lambda$ is generated by a family
of partial isometries $\{s_\lambda \colon \lambda \in \Lambda\}$
satisfying the Cuntz-Krieger relations of
\cite[Definition~1.5]{KumPas00}; we  point the reader to that reference for details.    The \emph{diagonal} subalgebra $D_\Lambda$ is the abelian
subalgebra generated by projections $\{s_\lambda s_\lambda^* \colon
\lambda \in \Lambda\}$.  Denote by $\Lambda^\infty$ the infinite path
space of $\Lambda$.  The \emph{cycline} subalgebra $M_{\Lambda}
\subseteq \cstar(\Lambda)$, generated by $\{s^{}_\lambda
s_\nu^* \colon \mu x = \nu x \text{ for all }x \in \Lambda^\infty\}$,  was introduced in \cite{BNR14} in order to
strengthen the Cuntz-Krieger Uniqueness Theorem.

The cycline algebra $M_{\Lambda}$ satisfies $D_\Lambda \subseteq
M_\Lambda$ and $M_\Lambda$ is abelian by \cite[Remark~7.2]{BNR14}.  By
\cite[Theorem~3.5]{Yan16}, $M_{\Lambda}$ is the relative commutant of
$D_\Lambda$ in $\cstar(\Lambda)$ and hence is maximal abelian.  As
$D_\Lambda$ is regular in $\cstar(\Lambda)$, $M_{\Lambda} \subseteq
\cstar(\Lambda)$ is also a regular inclusion, by
\cite[Lemma~2.10]{Pit17}.  Finally, $M_{\Lambda} \subseteq
\cstar(\Lambda)$ has the ideal intersection property by
\cite[Theorem~7.10]{BNR14}.  Thus, $(\cstar(\Lambda), M_\Lambda)$ is a
virtual Cartan inclusion and hence a pseudo-Cartan inclusion.  Note
that there may be no conditional expectation $E:
C^*(\Lambda)\rightarrow M_\Lambda$, see~\cite[Example~4.7]{BrownNagyReznikoffSimsWilliamsCaSuC*AlHaEtGr}.  Further analysis of ~\cite[Example~4.7]{BrownNagyReznikoffSimsWilliamsCaSuC*AlHaEtGr} is found in
Section~\ref{hrg} below.    
\end{example}

\begin{remark} Note that, in general, it is not the case that
  $(\cstar(\Lambda), D_{\Lambda})$ is pseudo-Cartan.  Indeed, if
  $\Lambda$ is a row-finite $k$-graph with no sources then, by
  \cite[Theorem~4.5]{KumPas00} and \cite[Proposition~5.5]{BCFS14},
  $(\cstar(\Lambda),D_\Lambda)$ has the \iip\ if and only if $\Lambda$
  satisfies Kumjian and Pask's aperiodicity condition
  (\cite[Definition~4.3]{KumPas00}).  In this case, $D_\Lambda$ is
  maximal abelian, and thus $D_\Lambda$ is a Cartan subalgebra of
  $\cstar(\Lambda)$.  Thus, in this setting,
  $(\ca(\Lambda),D_\Lambda)$ is a pseudo-Cartan inclusion if and only
  if it is a Cartan inclusion.
\end{remark}

\section{Isomorphisms of lattices of regular ideals} Let $A$ be a
$\ca$-algebra.  For a subset $X \subseteq A$, we denote by $X^{\perp}$
the set
$$ X^\perp = \{a \in A \colon aX = Xa = \{0\}\}. $$
The set $X^\perp$ is sometimes called the annihilator of $X$, e.g., in
\cite{Exel23, Vas23}, and $\mathrm{Ann}(X)$ is an alternative notation
for $X^\perp$.  If $(A,B)$ is an inclusion and $X \subseteq B$ we will
denote by $X^{\perp_B}$ and $X^{\perp\perp_B}$ the sets
\begin{align*} X^{\perp_B} &= X^{\perp} \cap B = \{ b \in B \colon bX
= Xb = \{0\}\}, \text{ and}\\ X^{\perp\perp_B} &= \{ b \in B \colon
bX^{\perp_B} = X^{\perp_B}b = \{0\}\} = (X^{\perp_B})^\perp \cap B.
\end{align*}

An ideal $J \unlhd A$ is a \emph{regular ideal} if $J =
J^{\perp\perp}$.  We refer the reader to \cite[Section~2]{BFPR24} for
a brief introduction to regular ideals.  By~\cite[Lemma~1.4]{Ham82}, the
collection of regular ideals in a $\ca$-algebra $A$ is 
a complete lattice, which we  denote   by
$\reg(A)$.

\begin{definition}\label{def: N inv ideal} Let $(A,B)$ be a regular
inclusion of $\ca$-algebras.  Let $N \subseteq N(A,B)$ be a generating
$*$-semigroup.  If $K \unlhd B$ we say that $K$ is called an
\emph{$N$-invariant ideal} if $nKn^* \subseteq K$ for all $n\in N$.
When $N = N(A,B)$ we will simply say that $K$ is an \emph{invariant
ideal}.  We denote all the regular invariant ideals in $B$ by
$\reginv_A(B)$.
\end{definition}

Let $(A,D)$ be a pseudo-Cartan inclusion and let $(A_1, D_1, \alpha)$
be a Cartan envelope for $(A,D)$.  Our goal in this section is to
prove Theorem~\ref{thm: all latice isomorphisms}, which says the
lattices $\reg(A)$, $\reginv_A(D)$, $\reg{A_1}$, and
$\reginv_{A_1}(D_1)$ are all isomorphic.  This theorem is broken into
three main steps.
\begin{enumerate}
    \item $\reg(A)$ and $\reginv_A(D)$ are isomorphic by
      \cite[Corollary~4.5]{Exel23}.  With additional hypotheses,
      this isomorphism is studied in more
detail in  Theorem~\ref{thm: 1-1 reg ideals pseudo}.
    \item An isomorphism between $\reginv_A(D)$ and
$\reginv_{A_1}(D_1)$ is given in Theorem~\ref{thm: D lattice
isomorphism}.
    \item The isomorphism between $\reg(A_1)$ and $\reginv_{A_1}(D_1)$
follows from \cite{BFPR24} (or \cite{Exel23}).
\end{enumerate}

Before explaining these isomorphisms, we establish the following useful
lemma.  Let $A$ be a $\ca$-algebra with injective envelope
$(I(A),\iota)$.  For regular ideals, Hamana shows more than just
the existence of the structure projection $p$.  We record this
information now.

\begin{lemma}[{cf.~\cite[Lemma~1.3]{Ham82}}]\label{lem: support
projections} Let $A$ be a $C^*$-algebra with injective envelope
$(I(A), \iota)$.  Let $J \unlhd A$ be a regular ideal and let $p \in
I(A)$ be the structure projection for $J$.  Then
$$ J = \{a \in A \colon \iota(a) \in pI(A)\}. $$
Further $pI(A) = \iota(J)^{\perp\perp}$.
\end{lemma}

\begin{proof} Let $J\unlhd A$ be a regular ideal.  The existence of
  the central projection $p \in I(A)$ satisfying
  $\iota(J) = pI(A) \cap \iota(A)$ follows from
  \cite[Lemma~1.3~(iii)]{Ham82} and \cite[Theorem~6.3]{Ham81}.  Since
  $p^\perp I(A)=\iota(J)^\perp$, we obtain
  $pI(A) = \iota(J)^{\perp\perp}$.
\end{proof}

\subsection{Regular ideals and faithful pseudo-expectations} The
following definition was introduced in \cite{Exel23}.

\begin{definition}[{\cite[Definition~3.3]{Exel23}}]\label{ExelInv} Let $(A,B)$ be
an inclusion of $\ca$-algebras.  The inclusion $(A,B)$ satisfies
\emph{condition \inv} if for all $J \unlhd A$
\begin{equation*} \overline{\spn}\{ab \colon a \in A,\ b \in J \cap
B\} = \overline{\spn}\{ba \colon a \in A,\ b \in J \cap B\}.
\end{equation*}
\end{definition}

Condition \inv\  is satisfied by the two main classes of inclusions
that we are interested in: abelian inclusions and regular inclusions.

\begin{example}\label{INVex} Let $(A,B)$ be an inclusion of $\ca$-algebras.
    \begin{enumerate}
        \item \label{INVex1} If $A$, and hence $B$, are abelian $\ca$-algebras then
$(A,B)$ satisfies \inv;
        \item \label{INVex2} If $(A,B)$ is a regular inclusion then $(A,B)$ satisfies
\inv\  by \cite[Proposition~4.2]{Exel23}.
    \end{enumerate}
\end{example}

In \cite[Theorem~3.24]{BFPR24} it was shown that if $(A,B)$ is a
regular inclusion satisfying the \iip\ and there is a faithful
conditional expectation satisfying an invariance condition, then $J
\mapsto J \cap B$ gives a lattice isomorphism from $\reg(A)$ to
$\reginv_A(B)$.  This result was greatly generalized by Exel:
\cite[Theorem~3.5]{Exel23} shows that if
$(A,B)$ is an inclusion satisfying the \iip\ and \inv, then the map
$J \mapsto J \cap B$ gives a lattice isomorphism from $\reg(A)$ to
$\reginv_A(B)$.

This leads to the following question.  Let $(A,B)$ be an inclusion of
$\ca$-algebras.  Given $J \unlhd A$, when is $J \cap B$ a regular
ideal in $B$?  Answering this question is needed to practically apply
\cite[Theorem~3.8]{PitZar15}.  In \cite[Proposition~3.7]{BFPR24} it is
shown that if there is a faithful conditional expectation $E \colon A
\rightarrow B$ then $J \cap B$ is a regular ideal in $B$ whenever $J$
is a regular ideal in $A$.  In \cite[Corollary~3.4(ii)]{Exel23} it is
shown that if $(A,B)$ satisfies the \iip\ and \inv,  then $J \cap B$
is a regular ideal in $B$ whenever $J$ is a regular ideal in $A$.  By
Proposition~\ref{prop: iip and faithful pseudo}, all
pseudo-expectations being faithful implies the \iip.  We now show that
a single faithful pseudo-expectation is all that is necessary for $J
\cap B$ to be regular when $J \unlhd A$ is regular, even in the
absence of \inv.

\begin{proposition}\label{prop: int reg pseudo} Let $(A,B)$ be an
inclusion of $C^*$-algebras, and let $(I(B),\iota)$ be an injective
envelope for $B$.  Assume there is a faithful pseudo-expectation $\Phi
\colon A \rightarrow I(B)$.  If $J\unlhd A$ is a regular ideal, then
$J \cap B$ is a regular ideal in $B$.
\end{proposition}

\begin{proof} First, let $J \unlhd A$ be an arbitrary ideal and let
$$ L  = \{ b \in B \colon \iota(b) \in \Phi(J)^\perp\}.$$
Then $L$ is a closed ideal in $B$.  If $b \in L$ and $a \in J$, then
$$ \Phi(b^*a^*ab) = \iota(b)^* \Phi(a^*a) \iota(b) =0. $$
Since $\Phi$ is faithful, it follows that $ab=0$.  Similarly $ba = 0$,
and so $L \subseteq J^\perp \cap B$.  Conversely, if $b \in J^\perp
\cap B$, then for any $a \in J$
$$\iota(b)\Phi(a) = \Phi(ba) = 0 \text{ and }\Phi(a)\iota(b)=\Phi(ab) = 0.$$
And so $b \in L$.  Thus, $L = J^\perp \cap B$.

Let $p \in I(B)$ be the structure projection for $L$ and let
$(b_\lambda)_\lambda$ be a positive, increasing approximate unit for
$L$.
Then, by \cite[Lemma~1.1]{Ham82}, $p = \sup_\lambda
\iota(b_\lambda)$.  Now take any $a \in J$.  By
\cite[Proposition~2.1.10]{SaiWriBook}
$$ \Phi(a)p\Phi(a)^* =\sup_\lambda \Phi(a) \iota(b_\lambda) \Phi(a)^* =0.$$
Hence $p \in \Phi(J)^\perp$, and therefore
$$ L = \{b \in B \colon \iota(b) = p\iota(b)\}.$$
Thus $L = J^\perp \cap B$ is a regular ideal in $B$ by
\cite[Lemma~1.3]{Ham82}.

Now take a regular ideal $J \unlhd A$.  Then, by the above argument,
$$ J \cap B = (J^\perp)^\perp \cap B$$
is a regular ideal in $B$.
\end{proof}

\begin{definition}\label{def: N inv psuedo} Let $(A,B)$ be a regular
inclusion of $\ca$-algebras, and let $(I(B), \iota)$ be an injective
envelope for $B$.  Let $N \subseteq N(A,B)$ be a generating
$*$-semigroup.  A pseudo-expectation $\Phi \colon A \rightarrow I(B)$
is \emph{$N$-invariant} if for $n \in N$ and $a \in A$
$$ \Phi(n^*an) = \tilde{\theta}_n(\Phi(|n^*|\, a\, |n^*|)),$$
where $\tilde{\theta}_n$ is the extension of the map
$|n^*|\, b\, |n^*| \mapsto n^*b n$ given in Lemma~\ref{lem: theta
  maps}.\footnote{For the equality
  $\Phi(n^*an) = \tilde{\theta}_n(\Phi(|n^*|\, a\, |n^*|))$ in
  Definition~\ref{def: N inv psuedo} to make sense,
  $\Phi(|n^*|\, a\, |n^*|)$ must belong to the domain of $\tilde\theta_n$.
  To see this, note that $(|n^*|^{1/k})$ is an approximate unit for
  $|n^*|\, B\, |n^*|$, so \cite[Lemma~1.1(i)]{Ham82} shows the structure
  projection $p$ for $|n^*|\, B\, |n^*|$ satisfies
  $p=\sup_{I(B)_{sa}} \iota(|n^*|^{1/k})$.  It follows that for every
  $a\in A$,
  $\Phi(|n^*|\, a\, |n^*|)=\iota(|n^*|)\Phi(a)\iota(|n^*|) \in pI(B) p
  =\text{dom}\, \tilde\theta_n$.  \black}

When $N = N(C,B)$,
i.e. when $\Phi$ is $N(C,B)$-invariant, we will simply call $\Phi$ an
\emph{invariant} pseudo-expectation.
\end{definition}

\begin{example} Let $(A,D)$ be a regular inclusion with $D$ abelian
and $A$ unital.  If there is a \emph{unique} pseudo-expectation $\Phi
\colon A \rightarrow I(D)$, then $\Phi$ is invariant by
\cite[Proposition~6.2]{Pit21}.
\end{example}

\begin{example} Let $(A,B)$ be an inclusion of $\ca$-algebras, let
$(I(B), \iota)$ be an injective envelope for $B$, and let $N \subseteq
N(A,B)$ be a $*$-semigroup.  Suppose that $\Phi \colon A \rightarrow
B$ is a conditional expectation such that $\iota \circ \Phi$ is an
$N$-invariant pseudo-expectation.  Then for any $a \in A$ and $n \in
N$
\begin{align*} \iota(\Phi(n^*an)) &=
\tilde{\theta}_n(\iota(\Phi(|n^*|\, a\, |n^*|))) \\ &=
\iota(\theta_n(|n^*|\Phi(a)|n^*|)) \\ &= \iota(n^* \Phi(a) n),
\end{align*} and so $\Phi(n^* a n) = n^* \Phi(a) n$.  Hence, for
conditional expectations, Definition~\ref{def: N inv psuedo} implies
\cite[Definition~3.13]{BFPR24}.  Thus all the examples in
\cite[Examples~3.15]{BFPR24} are examples of $N$-invariant
pseudo-expectations.
\end{example}

As previously stated, if $(A,B)$ is a regular inclusion of
$\ca$-algebras with the \iip, then $J \mapsto J\cap B$ is a lattice
isomorphism from $\reg(A)$ to $\reginv_A(B)$, by
\cite[Corollary~4.5]{Exel23}.  For the remainder of this subsection we
work towards a description of the inverse of this isomorphism when
there is a faithful $N$-invariant pseudo-expectation $\Phi \colon A
\rightarrow I(B)$.  This will culminate in Theorem~\ref{thm: 1-1 reg
ideals pseudo}, which will be used to study quotients of regular
ideals in Section~\ref{sec: quotients}.

\begin{notation}\label{not: J_K pseudo}
  Let $(A,B)$ be a regular inclusion of $\ca$-algebras with
  $(I(B),\iota)$ an injective envelope for $B$, and let
  $\Phi \colon A \rightarrow I(B)$ be a pseudo-expectation.  Let
  $K \unlhd B$ be an ideal.  Denote by $J_K$ the set
$$ J_K = \{ a \in A \colon \Phi(a^*a)\in \iota(K)\}, $$
and denote by $L_K$ the set
$$ L_K = \{ a \in A \colon \Phi(a^*a)\in \iota(K)^{\perp\perp}\}. $$
\end{notation}

In \cite[Proposition~3.19]{BFPR24}, it is shown that if $K \unlhd B$ is
an $N$-invariant regular ideal and $\Phi$ is a faithful $N$-invariant
conditional expectation, then $J_K$ is a regular ideal of $B$;
furthermore, in this
context,  $J_K=L_K$.    
When the hypothesis that there is an $N$-invariant
conditional expectation is relaxed to assuming there is an $N$-invariant
pseudo-expectation, we shall next show $J_K$ is a \emph{right} ideal,
but we make no claim on $J_K$ being a left ideal.  This is because we
do not know whether $\iota(K)$ is hereditary in $I(B)$ for an ideal
$K \unlhd B$.

\begin{proposition}\label{prop: JK ideal pseudo} Let $(A,B)$ be a
regular inclusion of $C^*$-algebras and let $(I(B),\iota)$ be an
injective envelope for $B$.  Let $N \subseteq N(A,B)$ be a generating
$*$-semigroup.  Assume that there is an $N$-invariant
pseudo-expectation $\Phi \colon A \rightarrow I(B)$ and let $K \unlhd B$
be an $N$-invariant ideal.  The following statements hold.
\begin{enumerate}[(i)]
\item $J_K$ is a right ideal in $A$ satisfying $J_K \cap B = K$.
\item $L_K$ is an ideal in $A$ satisfying $L_K \cap B =
K^{\perp\perp_B}$.
\item If $\Phi$ is faithful and  $K$ is in addition assumed to be a
  regular ideal,
  then $L_K$ is a regular ideal of $A$ satisfying $L_K =
J_K^{\perp\perp}$.
\end{enumerate}
\end{proposition}

\begin{proof} (i) For $n \in N$, recall the $*$-isomorphisms
  $\theta_n: \overline{n B n^*}\rightarrow \overline{n^* B n}$ and
  $\tilde{\theta}_n \colon I(\overline{n B n^*}) \rightarrow
  I(\overline{n^* B n})$ from Lemma~\ref{lem: theta maps}.

Given $a\in J_K$,
we can find $k\in K$ such that $\Phi(a^*a)=\iota(k)$.  Fix $n \in
N$ and let $$c = |n^*| = (nn^*)^{1/2}.$$  Since $\Phi$ is
$N$-invariant  and $c\in B$,
\begin{align*}\Phi(n^*a^*an)&=\tilde\theta_n(\Phi(ca^*ac))=\tilde\theta_n(\iota(c)\Phi(a^*a)\iota(c))=\tilde\theta_n(\iota(c)\iota(k)\iota(c))\\
&=\tilde\theta_n(\iota(ckc)) = \iota(\theta_n(ckc))=\iota(n^*kn).
\end{align*} Since $K$ is $N$-invariant, we have $\Phi(n^*a^*an)\in
\iota(K)$, and so $an\in J_K$.  As the $*$-semigroup $N$ spans a dense
subset of $A$, we conclude $J_K$ is a right ideal in $A$.

Let $b \in B$. Then
$$ b \in J_K \iff \Phi(b^*b) \in \iota(K).$$
Since $\Phi(b^*b) = \iota(b^*b)$, and $\iota$ is injective, it follows
that $b \in J_K$ if and only if $b \in K$.  Hence, $J_K \cap B = K,$
completing the proof of (i).

(ii) We will now show that $L_K$ is an ideal.  To show that $L_K$ is a
left ideal we will follow similar reasoning as in
\cite[Proposition~3.19(i)]{BFPR24}.  Take $a \in L_K$ and $b \in A$.
Then $a^*b^*ba \leq \|b\|^2 a^*a$ and hence
$$\Phi((ba)^* ba) = \Phi(a^*b^*ba) \leq \|b\|^2 \Phi(a^*a) \in \iota(K)^{\perp\perp},$$
since pseudo-expectations are (completely) positive maps.  As ideals
in $C^*$-algebras are hereditary, it follows that $\Phi((ba)^*(ba))
\in \iota(K)^{\perp\perp}$.   Hence, $ba \in L_K$, and so $L_K$ is a
left-ideal of $A$.

Note that, for $k \in K$ and $n \in N$, $\theta_n(|n^*|\, k\, |n^*|) = n^*kn
\in K$, since $K$ is $N$-invariant.  Hence $\theta_n$ restricts to a
$*$-isomorphism
$$ \theta_n^K \colon \overline{|n^*| K |n^*|} \longrightarrow \overline{n^* K n}.$$
Let $q_1$ be the structure projection for $\overline{|n^*|K|n^*|}$ and
let $q_2$ be the structure projection for $\overline{n^*Kn}$.  Then
$\theta_n^K$ extends uniquely to a $*$-isomorphism
$$ \widetilde{\theta_n^K} \colon q_1 I(B) q_1 \rightarrow q_2 I(B) q_2$$
satisfying $\widetilde{\theta_n^K} \circ \iota = \iota \circ
\theta_n^K.$ As $\theta_n^K$ and $\theta_n$ agree on
$\overline{|n^*|K|n^*|}$, the uniqueness of $\widetilde{\theta_n^K}$
says that $$\widetilde{\theta_n^K} = \tilde\theta_n|_{q_1I(B)q_1}.$$
In particular, if $x \in \iota(K)^{\perp\perp}=I(K)$ (the equality is
from Lemma~\ref{lem: support projections}) and $x$ is in the
domain of $\tilde{\theta}_n$, then $\tilde{\theta}_n(x) \in
\iota(K)^{\perp\perp}$.

Take any $a \in L_K$.  Using again the fact that $\Phi$ is
$N$-invariant, we get
\begin{align*}\Phi(n^*a^*an)&=\tilde\theta_n(\Phi(|n^*|\, a^*a\, |n^*|))=\tilde\theta_n(\iota(|n^*|)\Phi(a^*a)\iota(|n^*|))\in
 \iota(K)^{\perp\perp}.\end{align*} Hence, $an \in L_K$.  Since the $*$-semigroup $N$ spans a
dense subset of $A$, we conclude $L_K$ is an ideal in $A$.

Note that $b \in L_K \cap B$ if and only if $\Phi(b^*b) \in
\iota(K)^{\perp\perp}$.  However, as $\Phi|_B = \iota$, and $\iota$ is
injective, it follows that $\iota(b) \in \iota(K)^{\perp\perp} \cap
\iota(B)$.  As $\iota(K)^{\perp\perp}\cap \iota(B) =
\iota(K^{\perp\perp})$ by Lemma~\ref{lem: support projections}, we
have  $L_K \cap B = K^{\perp\perp}.$ This completes the proof of
(ii).

(iii) Now assume that $K \unlhd B$ is an $N$-invariant regular ideal
and that $\Phi$ is a faithful $N$-invariant pseudo-expectation.
Since $\iota(K^\perp_B)^{\perp\perp} = \iota(K)^\perp$,
Lemma~\ref{lem: support projections} gives
$$ L_{K^\perp} = \{a \in A \colon \Phi(a^*a) \in \iota(K)^\perp \}. $$

If $a \in J_K^\perp$ and $b \in K$, we have $\Phi(a^*a)\iota(b)=
\Phi(a^*ab) = 0,$ and $\iota(b)\Phi(a^*a)= \Phi(ba^*a) = 0.$ Hence, if
$a \in J_K^\perp$, $\Phi(a^*a) \in \iota(K)^\perp$.  Thus $J_K^\perp
\subseteq L_{K^\perp}$.

Conversely, if $a \in L_{K^\perp}$ and $c \in J_K$ then $ca \in J_K
\cap L_{K^\perp}$.  Hence $\Phi(a^*c^*ca) \in \iota(K) \cap
\iota(K)^\perp = \{0\}$.  As $\Phi$ is faithful, it follows that $ca =
0$.

To show that $ac = 0$, note that
$$ \Phi(c^*a^*ac) \leq \|a\|^2 \Phi(c^*c) \in \iota(K). $$
Hence, $\Phi(c^*a^*ac)$ lies in the ideal generated by $\iota(K)$ in
$I(B)$.  In particular, $\Phi(c^*a^*ac) \in \iota(K)^{\perp\perp}.$ As
$a \in L_{K^\perp}$, we also have that $ac \in L_{K^{\perp}}$.
Therefore,
$$\Phi (c^*a^*ac) \in \iota(K)^{\perp\perp} \cap \iota(K)^\perp = \{0\}.$$
Thus $J_K^\perp = L_{K^\perp}$.  It follows that
$$ J_K^{\perp\perp} = L_{K^\perp}^\perp = \{a \in A \colon \Phi(a^*a) \in \iota(K)^{\perp\perp} \} = L_K. $$
Hence $L_K$ is regular and (iii) is proved.
\end{proof}

Let $(A,B)$ be a regular inclusion satisfying the \iip, let
$(I(B),\iota)$ be an injective envelope for $B$, and let
$\Phi \colon A \rightarrow I(B)$ be a faithful pseudo-expectation.  By
\cite[Corollary~4.5]{Exel23}, the map $J \mapsto J \cap B$ is a lattice
isomorphism from $\reg(A)$ to $\reginv_A(B)$, with inverse map
$K \mapsto K^{\perp\perp_A}$.  In this context, an alternate
description for the inverse map will be useful for us in Section 4,
when we study quotients by regular ideals.  Theorem~\ref{thm: 1-1 reg
  ideals pseudo} gives the alternate description; note that this
theorem generalizes \cite[Theorem~3.24]{BFPR24}.

\begin{theorem}\label{thm: 1-1 reg ideals pseudo}
  Let $(A,B)$ be a regular inclusion of $\ca$-algebras satisfying the
  \iip, let $(I(B),\iota)$ be an injective envelope for $B$, and let
  $N$ be a generating $*$-semigroup. Assume there is an $N$-invariant
  faithful pseudo-expectation $\Phi \colon A \rightarrow I(B)$.

  The invariant regular ideals of $B$ form a Boolean algebra.  The map
  $J \mapsto J \cap B$ has inverse given by
  $K \mapsto L_K = J_K^{\perp\perp} = K^{\perp\perp}$ and is a
  Boolean algebra isomorphism between the regular ideals of $A$ and
  the invariant regular ideals of $B$.
\end{theorem}

\begin{proof}
  That the invariant regular ideals of $B$ form a Boolean algebra
  follows as in the proof of \cite[Theorem~3.24]{BFPR24}.

Let $J \unlhd A$ be a regular ideal.  Let $K = J \cap B$.  Then,
$K\unlhd B$ is an invariant ideal.  By Proposition~\ref{prop: int reg
  pseudo}, $K$ is also a regular ideal.  By Proposition~\ref{prop: JK
  ideal pseudo}, $L_K = J_K^{\perp\perp}$ is regular ideal of $A$
satisfying $L_K \cap B = K = J \cap B$.

Suppose $J \not\subseteq L_K$.  Then there is $a \in J$ and 
$b \in L_K^\perp$ such that $ab \neq 0$.  Hence
$J \cap L_K^{\perp} \neq \{0\}$.  However, $J \cap L_K^\perp$ is a
regular ideal having trivial intersection with $B$, contradicting the
\iip.  Hence $J \subseteq L_K$.

If $J \neq L_K$, then the regular ideal $L = L_K \cap J^\perp$ is
non-zero.  Note that $L \subseteq L_K$ and $L \cap J = \{0\}$.  Hence
$L \cap B = \{0\}$.  This contradicts the \iip, and hence $J = L_K$.
It follows that the map $J \mapsto J \cap B$ is one-to-one with
inverse $K \mapsto L_K$.  That $L_K = K^{\perp\perp}$ follows from
\cite[Theorem~3.5]{Exel23}.
\end{proof}

\begin{remark}\label{rem: iip riip}
  The proof of Theorem~\ref{thm: 1-1 reg ideals pseudo} does not use
  the full power of the \iip.  Instead, it is only necessary to assume
  that all non-zero regular ideals of $A$ have non-trivial
  intersection with $B$.  That is, Theorem~\ref{thm: 1-1 reg ideals
    pseudo} could be stated with the assumption that $(A,B)$ has the
  \emph{\riip} in place of the \iip.  The difference (or agreement) of
  the \iip\ and the \riip\ is addressed in more detail in
  \cite[Section~7]{BFPR24}.
\end{remark}

\begin{theorem}\label{PseudoCarLatIso}
  If $(A,D)$ is a pseudo-Cartan inclusion then the Boolean algebras
  $\reg(A)$ and $\reginv_A(D)$ are isomorphic via the map
  $J \mapsto J \cap D$, with inverse given by
  $K \mapsto L_K = J_K^{\perp\perp} = K^{\perp\perp}$.
\end{theorem}

\begin{proof}
  A pseudo-Cartan inclusion has the \iip, and hence the \riip.  By
  Theorem~\ref{thm: pseudo-Cartan equiv}, $(A,D)$ has the faithful
  unique pseudo-expectation property.  Let
  $\Phi \colon A \rightarrow I(D)$ be the faithful unique
  pseudo-expectation and let $(\tilde{A}, \tilde{D})$ be the unitization
  of $(A,D)$.  Then $(\tilde{A}, \tilde{D})$ has the unique
  pseudo-expectation property by Theorem~\ref{thm: pseudo-Cartan
    equiv}, and
  $\tilde{\Phi} \colon \tilde{A} \rightarrow I(\tilde{D})$ is the
  faithful unique pseudo-expectation.  By
  \cite[Proposition~6.2]{Pit21}, $\tilde{\Phi}$ is invariant.  Since
  $D$ contains an approximate unit for $A$,
  $N(A,D) \subseteq N(\tilde{A},\tilde{D})$, and hence $\Phi$ is
  invariant.  The result thus follows immediately from
  Theorem~\ref{thm: 1-1 reg ideals pseudo}.
\end{proof}

\subsection{Regular ideals and Cartan envelopes}
Let $(A,D)$ be a pseudo-Cartan inclusion with Cartan envelope
$(A_1, D_1, \alpha)$.  We wish to prove  the
lattices $\reg(A)$, $\reginv_A(D)$, $\reg(A_1)$, and
$\reginv_{A_1}(D_1)$ are all isomorphic.  To complete this it suffices to prove the
$\reginv_A(D)$ and $\reginv_{A_1}(D_1)$ are isomorphic (note that
$\reg(A) \simeq \reginv_A(D)$, and
$\reg(A_1) \simeq \reginv_{A_1}(D_1)$ by Theorem~\ref{thm: 1-1 reg
  ideals pseudo} or \cite[Theorem~3.5]{Exel23}).

If $(A,B)$ is a regular inclusion of $\ca$-algebras, and $J \unlhd B$ is an invariant ideal, then \cite[Proposition~4.2]{Exel23} shows that
$$ \overline{\spn}\{ ax \colon a \in A,\ x\in J \} = \overline{\spn}\{
xa \colon a \in A,\ x\in J \}\quad\text{(cf.\ Definition~\ref{ExelInv})}.$$
When $B$ is an abelian $\ca$-algebra we can give a more explicit
description of this intertwining property.  We do this in the
following lemma, which will be useful in proving Theorem~\ref{thm: D
  lattice isomorphism}.

\begin{lemma}\label{lem: intertwine ideals}
  Let $(A,D)$ be an inclusion of $\ca$-algebras, with $D$ abelian.  If
  $J \unlhd D$ is an ideal and $n \in N(A,D)$ satisfies
  $nJn^* \subseteq J$ and $n^*Jn \subseteq J$, then
    \begin{enumerate}
    \item $Jn \subseteq \overline{nJ}$, and $nJ \subseteq \overline{Jn}$; 
    \item $nJ^\perp n^* \subseteq J^\perp$ and $n^*J^\perp n \subseteq J^\perp$.
    \end{enumerate}
 In particular, if $J \unlhd D$ is an invariant ideal then $J^\perp$ is an invariant ideal.
\end{lemma}

\begin{proof}
    Take $a \in J$. Then
    $$ an = \lim_k a (nn^*)^{1/k}n = \lim_k (nn^*)^{1/k}an.$$
    Note that, for all $k$, $(nn^*)^{1/k}a \in \overline{nn^* J}.$
    Thus $an \in \overline{nn^* J n}\subseteq \overline{nJ}.$
    A similar argument shows that $nJ \subseteq \overline{Jn}$.

    Now take any $b \in J^\perp$ and $a \in J$.
    Choose a sequence $(h_k)$ in $J$ such that $an = \lim_k n h_k$.
    Then 
    $$ a (n b n^*) = \lim_k (n h_k) b n^* =  0.$$
    Thus $n J^\perp n^* \subseteq J^\perp$.
    It can be similarly be shown that $n^* J^\perp n \subseteq J^\perp$.
\end{proof}

\begin{theorem}\label{thm: D lattice isomorphism}
  Let $(A,D)$ be a pseudo-Cartan inclusion with Cartan envelope
  $(A_1,D_1,\alpha)$.  The map $J \mapsto \alpha(J)^{\perp\perp}$ is
  an isomorphism from the lattice $\reginv_A(D)$ of invariant regular
  ideals of $D$ to the lattice $\reginv_{A_1}(D_1)$ of invariant
  regular ideals of $D_1$.  The inverse map is given by
  $K \mapsto \alpha^{-1}(K)$.
\end{theorem}

\begin{proof}
  Since $\alpha(D)$ has the \iip\ in $D_1$, \cite[Theorem~3.5]{Exel23}
  shows that the map $J \mapsto \alpha(J)^{\perp\perp}$ is an
  isomorphism from the lattice of regular ideals in $D$ to the regular
  ideals in $D_1$.  The inverse of this map is given by
  $\alpha^{-1}(K)$ for a regular ideal $K$ in $D_1$.  We will show
  that the restriction of this map to the invariant regular ideals of
  $D$ maps onto the invariant regular ideals of $D_1$.

Let $J \unlhd D$ be a regular, invariant ideal and put
 $$J_1:= \alpha(J)^{\perp\perp}.$$
Then $J_1$ is a regular ideal in $D_1$.    We must show $J_1$ is an
invariant ideal.
To do this, take $k \in \alpha(J)^\perp$, $n \in N(A,D)$ and $a \in J$.  By
Lemma~\ref{lem: intertwine ideals}, there is a sequence $(h_m)_m$ in
$J$ such that $na = \lim_m h_m n$.  Then
\begin{align*}
    \alpha(n)^* k \alpha(n) \alpha(a) & =  \alpha(n)^* k \alpha(na) \\
    &= \lim_m \alpha(n)^* k \alpha(h_m n) \\
    &= \lim_m \alpha(n)^* k \alpha(h_m) \alpha(n) \\
    &= 0.
\end{align*}
Thus $\alpha(n)^* k \alpha(n) \alpha(J)=0$.  Similarly,
$\alpha(J)\alpha(n)^* k \alpha(n)=0$.  Thus, for all $n \in N(A,D)$,
$$ \alpha(n)^* \alpha(J)^\perp \alpha(n)^* \subseteq \alpha(J)^\perp.$$

Now take $a_1 \in J_1$, $x \in \alpha(J)^\perp$ and $n \in N(A,D)$.
By Lemma~\ref{lem: intertwine ideals}, there is a sequence $(x_m)_m$ in $\alpha(J)^\perp$ such that $\alpha(n)x = \lim_m x_m \alpha(n)$.
Thus
\begin{align*}
    \alpha(n)^* a_1 \alpha(n) x &= \lim_m \alpha(n)^* a x_m \alpha(n)^* \\ 
    &= 0.
\end{align*}
Hence $J_1$ is invariant for $\alpha(n)$ for all $n \in N(A,D)$.  As
$J_1 \unlhd D_1$, we also have that $dJ_1d^* \subseteq J_1$ for all
$d \in D_1$.  Hence $J_1$ is invariant under the $*$-semigroup
generated by $\alpha(N(A,D))\cup D_1$.  We call this semigroup $M$.
By definition of the Cartan envelope, $\spn M$ is dense in $A_1$.
Hence, by
\cite[Corollary~3.20]{BFPR24}, $J_1$ is an invariant ideal.

Next we show that the inverse map, $ J_1\mapsto \alpha^{-1}(J_1)$,
carries invariant regular ideals of $D_1$ to invariant regular ideals
in $D$.  Fix an invariant regular ideal $J_1 \unlhd D_1$ and let
$J = \alpha^{-1}(J_1)$.  From the definition of Cartan envelope,
$(D_1, \alpha(D))$ has the \iip, 
and $(D_1, \alpha(D))$ satisfies
\inv\  by Example~\ref{INVex}\eqref{INVex1}.
By~\cite[Theorem~3.5]{Exel23}, $J$ is a
regular ideal in $D$.
It remains to show that $J$ is invariant.
For $n \in N(A,D)$, 
\begin{align*}
    \alpha(nJn^*) &= \alpha(n) \alpha(J) \alpha(n)^* \\
    & \subseteq \alpha(J)^{\perp\perp}=J_1.
\end{align*}
Therefore $nJn^* \subseteq J$, and hence $J$ is an invariant ideal.
\end{proof}

We now have all the ingredients to prove the main theorem of this subsection.

\begin{theorem}\label{thm: all latice isomorphisms}
Let $(A,D)$ be a pseudo-Cartan inclusion with Cartan envelope $(A_1,D_1,\alpha)$.
Then the following lattices of ideals are isomorphic
$$  \reg(A) \simeq \reginv_A(D) \simeq \reginv_{A_1}(D_1) \simeq \reg(A_1)$$ 
\end{theorem}

\begin{proof}
  The result follows immediately from Theorem~\ref{thm: 1-1 reg ideals
    pseudo} (or \cite[Corollary~4.5]{Exel23}) and Theorem~\ref{thm: D
    lattice isomorphism}.  Explicitly, the isomorphisms are given by:
\begin{align*}
    \Theta_1  \colon & \reg(A) \longrightarrow \reginv_A(D) \\
    &  J \longmapsto J \cap D \\
    \Theta_2  \colon  & \reginv_A(D) \longrightarrow \reginv_{A_1}(D_1) \\
    &  K \longmapsto \alpha(K)^{\perp\perp_{D_1}} \\
    \Theta_3  \colon  & \reginv_{A_1}(D_1) \longrightarrow \reg(A_1) \\
    &  K \longmapsto L_K = K^{\perp\perp}.\qedhere
\end{align*}
\end{proof}

\section{Quotients by regular ideals}\label{sec: quotients}
Suppose $(A,B)$ is a unital inclusion and $J\idealin A$ is such that
$J\cap B\in \reg(B)$.  By~\cite[Theorem~3.8]{PitZar15}, if $(A,B)$ has
the unique pseudo-expectation property, so does the quotient inclusion
$(A/J, B/(B\cap J))$.  However, \cite[Example~4.3]{PitZar15} shows
that it is possible that $(A/J, B/(J\cap B))$ does not have a
\emph{faithful} unique pseudo-expectation, even if the
pseudo-expectation for $(A,B)$ is faithful.  The purpose of this
section is to explore the behaviour of the faithful unique
pseudo-expectation property under the quotient  by an ideal $J$
satisfying various hypotheses. 

In Lemma~\ref{lem: PZ
  quotient} we extend~\cite[Theorem~3.8]{PitZar15}
to not-necessarily unital inclusions.  Then in
Theorem~\ref{thm: when quot has faithful} we give necessary and
sufficient conditions for $(A/J,B/(J\cap B))$ to have the faithful
unique pseudo-expectation property, provided $(A,B)$ has the faithful unique
pseudo-expectation property.  It follows that the faithful unique
pseudo-expectation property is preserved by quotients by regular
ideals in both regular inclusions and inclusions of abelian
$\ca$-algebras.
We end the section by returning to pseudo-Cartan inclusions.  In
Theorem~\ref{thm: pseudo cartan quotient} we show that pseudo-Cartan
inclusions are preserved by quotients by regular ideals.  Further, the
Cartan envelope of the quotient is a quotient of the original Cartan
envelope.

The proof of the following lemma follows the same reasoning as
\cite[Theorem~3.8]{PitZar15}.  The proof is included here for
completeness.  Note that in \cite{PitZar15} the $C^*$-algebras are
assumed to be unital.  We do not make that assumption here.

\begin{lemma}\label{lem: PZ quotient}
  Let $(A,B)$ be an inclusion having the unique pseudo-expectation
  property and let $\Phi \colon A \rightarrow I(B)$ be the
  pseudo-expectation.  If $J \unlhd A$ is an ideal such that
  $J \cap B$ is a regular ideal in $B$ with structure projection $p$,
  then $(A/J, B/(J \cap B))$ has the unique pseudo-expectation
  property.  Furthermore, the pseudo-expectation for
  $(A/J, B/(J\cap B))$ is given by $a+J\mapsto \Phi(a) p^\perp$.
\end{lemma}

\begin{proof}
  Let $$K = J \cap B$$ and let $p$ be the structure projection for $K$
  in $I(B)$.  Note that, by Lemma~\ref{lem: support projections},
  $K = \{b \in B \colon \iota(b)p = \iota(b)\}$.  Define a map
    \begin{equation*}
    \psi \colon B/ K  \longrightarrow \iota(B)p^\perp \quad \text{by} \quad
    b + K  \longmapsto \iota(b)p^\perp.
    \end{equation*}
    Note that $\psi$ is well-defined. 
    Indeed, if $b_1 = b_2 + k$ for some $b_1, b_2 \in B$ and $k \in K$ then
    $$ \iota(b_2)p^\perp = (\iota(b_2) + \iota(k))p^\perp = \iota(b_1)p^\perp.$$

    We now show $\psi$ is a $*$-isomorphism.
    Since $K=\{b\in B: \iota(b)p=\iota(b)\}$, $\psi$ is surjective.
    To see that $\psi$ is one-to-one, note that for $b\in B$,
\[\psi(b+K)=0\iff \iota(b)p^\perp=0\iff b\in K.\]

Thus $\psi$ extends to a $*$-isomorphism $\tilde{\psi}$
between the injective envelopes:
    $$ \tilde{\psi} \colon I(B/K) \rightarrow I( \iota(B)p^\perp) \simeq I(B)p^\perp.$$
    Therefore, $(I(B)p^\perp, \psi)$ is an injective envelope for
    $B/K$. 

    Let $\Phi \colon A \rightarrow I(B)$ be the unique pseudo-expectation, and let
    $$\theta \colon A/J \rightarrow I(B)p^\perp\simeq I(B/K)$$
    be any pseudo-expectation for the inclusion $(A/K, B/K)$.
    Note that, since $\theta$ is a pseudo-expectation,
    $$ \theta(b + J) = \psi(b+J) = \iota(b)p^\perp$$
    for any $b \in B$.
    Now define
    \begin{align*}
        \Theta \colon A &\longrightarrow I(B)\\
        a &\longmapsto \theta(a +J) + \Phi(a)p.
    \end{align*}
    Note that $\Theta$ is pseudo-expectation for the inclusion $B \subseteq A$.
    As $\Phi$ is the unique pseudo-expectation, we must have $\Theta = \Phi$.
    Thus, $\theta(a+J) = \Phi(a)p^\perp$ for all $a+J \in A/J$.
    That is, the pseudo-expectation $\theta$ is unique.
\end{proof}

We now consider the behavior of the faithful unique pseudo-expectation property under quotients. 

\begin{theorem}\label{thm: when quot has faithful}
  Suppose $(A,B)$ is a inclusion of $\ca$-algebras having the faithful
  unique pseudo-expectation property and let $\Phi$ be the pseudo-expectation.
  Let $J \unlhd A$ be an ideal such that $K:= J \cap B$  is a regular
  ideal of $B$.  Then $(A/J,B/K)$ has the faithful unique
  pseudo-expectation property if and only if
$$ J = L_K = \{a \in A \colon \Phi(a^*a) \in \iota(K)^{\perp\perp}\}.$$

In particular,  if $(A/J,B/K)$ has the faithful unique pseudo-expectation property, then $L_K$ must be an ideal of $A$.
\end{theorem}

\begin{proof}
    Note that $L_K$ need not be an ideal in $A$. (Since $(A,B)$ is not necessarily a regular inclusion, Proposition~\ref{prop: JK ideal pseudo} does not apply.)
    However, $L_K$ is a closed left-ideal of $A$.
     
Let $p\in I(B)$ be the structure projection for $K$ (recall $p$ is
central).  Since $\Phi$ is the unique pseudo-expectation for $(A,B)$
and $K$ is a
regular ideal of $B$, Lemma~\ref{lem: PZ quotient} shows
$(A/J, B/K)$ has a unique pseudo-expectation $\Psi$
 given by
    $$ \Psi(a + J) = \Phi(a)p^\perp. $$

    Let $N_\Psi$ be the left kernel of $\Psi$.
    That is,
    $$ N_\Psi = \{ a + J \in A/J \colon \Psi((a+J)^*(a+J)) = 0 \}.$$
    Then  $N_\Psi$ is a closed-left ideal and 
    $\Psi$ is faithful if and only if $N_\Psi = \{0\}$.
    
    Denote by $q$ the quotient map $q \colon A \rightarrow A/J$.
    Then $N_\Psi = \{0\}$ if and only if $q^{-1}(N_\Psi) = J$.
   As $q(J)=0$,  $J \subseteq q^{-1}(N_\Psi).$
    Then 
    \begin{align*}
        q^{-1}(N_\Psi) &= \{ a \in A \colon q(a) \in N_\Psi\} \\
        &= \{ a \in A \colon \Psi((a+J)^*(a+J)) = 0 \}\\
        &= \{ a \in A \colon \Phi(a^*a)p^\perp = 0\}\\
        &= \{ a \in A \colon \Phi(a^*a) \in I(B)p \}.
    \end{align*}
    By Lemma~\ref{lem: support projections}, $pI(B) = \iota(K)^{\perp\perp}$, and hence $q^{-1}(N_\Psi) = L_K.$
    Hence $\Psi$ is faithful if and only if $J = L_K$.
\end{proof}

\begin{corollary}\label{cor: when quot has faithful}
    Let $(A,B)$ be a regular inclusion of $\ca$-algebras.
    Assume that $(A,B)$ has the faithful unique pseudo-expectation property.
    If $J \unlhd A$ is a regular ideal, then $(A/J, B/(J\cap B))$ has the faithful unique pseudo-expectation property.
\end{corollary}

\begin{proof}
Note that $(A,B)$ has the \iip\ by Proposition~\ref{prop: iip and faithful pseudo}.
Let $J \unlhd A$ be a regular ideal.
By Theorem~\ref{thm: 1-1 reg ideals pseudo}, if $K = J \cap B$, then $J = L_K$.
Hence, by Theorem~\ref{thm: when quot has faithful}, $(A/J, B/(J\cap
B))$ has the faithful unique pseudo-expectation property.
\end{proof}

\begin{corollary}\label{cor: abelian quotients}
Let $(A,B)$ be an inclusion of abelian $\ca$-algebras.
Assume that $(A,B)$ has the faithful unique pseudo-expectation property.
If $J \unlhd A$ is a regular ideal, then $(A/J, B/(J\cap B))$ has the faithful unique pseudo-expectation property.
\end{corollary}

\begin{proof} Apply Lemma~\ref{abel+AUP}, then Corollary~\ref{cor:
    when quot has faithful}.
\end{proof}

If $(A,B)$ has the faithful unique pseudo-expectation property, then $(A,B)$ has the \iip, by Proposition~\ref{prop: iip and faithful pseudo}.
The following proposition shows that if $(A,B)$ is a regular inclusion then the \iip\ is also preserved under quotients by regular ideals.

\begin{proposition}\label{prop: quotient iip}
    Let $(A,B)$ be a regular inclusion of $\ca$-algebras and let $J \unlhd A$ be a regular ideal.
    If $(A,B)$ has the \iip, then $(A/J, B/(J\cap B))$ has the \iip.    
\end{proposition}

\begin{proof}
The proof follows the same reasoning as \cite[Theorem~4.2]{BFPR24}.
The proof is provided here for completeness and to highlight the role \cite{Exel23} plays in this setting.
Let $K = J \cap B$ and   suppose
 $I \unlhd A/J$ is an ideal satisfying $I \cap (B/K) = \{0\}.$
Let $q: A\rightarrow A/J$ be the quotient map and let $L = q^{-1}(I)$.
Note that $L \cap B = K$ and $J \unlhd L$.  We shall show $J=L$.
Suppose not.  Then
by \cite[Lemma~3.4]{BFPR24}, $J^\perp \cap L$ is a non-zero ideal.
By \cite[Theorem~3.5]{Exel23}, $J^\perp \cap B \subseteq K^\perp$.
Hence $(J^\perp \cap L) \cap B = \{0\}$.
By the \iip, $J^\perp \cap L = \{0\}$.  This contradiction
gives $J = L$.
It follows that $I = \{0\}$ and hence $(A/J,B/K)$ has the \iip.
\end{proof} 

We end this section with an application to pseudo-Cartan inclusions.

\begin{theorem}\label{thm: pseudo cartan quotient}
  Let $(A,D)$ be a pseudo-Cartan inclusion and let $J \unlhd A$ be a
  regular ideal. Then $(A/J, D/(J\cap D))$ is a pseudo-Cartan
  inclusion.

  Further, suppose $(A_1,D_1,\alpha)$ is a Cartan envelope for
  $(A,D)$, with conditional expectation
  $E \colon A_1 \rightarrow D_1$.  Let
$$ J_1 = \{a \in A_1 \colon E(a^*a) \in \alpha(J\cap D)^{\perp\perp_{D_1}}\}.$$
Then $J_1\in \reg(\A_1)$, $\alpha(J)\subseteq J_1$,
and if $\overline\alpha: A/J\rightarrow A_1/J_1$ is defined by
$\overline{\alpha}(a+J) = \alpha(a) + J_1$, then $\overline\alpha$ is
a $*$-monomorphism and $(A_1/J_1, D_1/(J_1\cap D_1), \overline{\alpha})$ is a Cartan envelope for $(A/J, D/(J\cap D))$.
\end{theorem}

\begin{proof}
Let $J \unlhd A$ be a regular ideal.
The inclusion $(A/J,D/(J\cap D))$ is regular since $(A,D)$ is a
regular inclusion.
Corollary~\ref{cor: when quot has faithful} shows $(A/J, D/(J\cap D))$
has the faithful unique pseudo-expectation property.  Hence
$(A/J, D/(J\cap D))$ is a pseudo-Cartan inclusion by Theorem~\ref{thm:
  pseudo-Cartan equiv}.

Let $(A_1, D_1, \alpha)$ be a Cartan envelope for $(A,D)$, and let $E \colon A_1 \rightarrow D_1$ be the faithful conditional expectation.
Let $(I(D_1),\iota)$ be an injective envelope for $D_1$.
By, \cite[Proposition~2.4.4]{Pit25}, $(I(D_1),\iota\circ \alpha)$ is an injective envelope for $D$.
Let
$$\Phi = \iota \circ E \circ \alpha \colon A \rightarrow I(D_1).$$
Then $\Phi$ is a pseudo-expectation for the inclusion $(A,D)$.
As $(A,D)$ is a pseudo-Cartan inclusion, $\Phi$ is the faithful unique pseudo-expectation for $(A,D)$.
 
Note that the map
$$ \Theta \colon L \mapsto \{a \in A_1 \colon E(a^*a) \in \alpha(L\cap D)^{\perp\perp_{D_1}}\}$$
is the isomorphism between $\reg(A)$ and $\reg(A_1)$ given by Theorem~\ref{thm: all latice isomorphisms}.
Observe also that $J_1 = \Theta(J)$.

Let $K = J \cap D$, and $K_1 = (\alpha(J\cap D))^{\perp\perp_{D_1}}$.
Let $p \in I(D_1)$ be the structure projection for $K_1$ (and hence $K$).
Take any $a \in A$.
Then, $\alpha(a) \in J_1$ if and only if $E(\alpha(a^*a)) \in K_1$.
Further $E(\alpha(a^*a)) \in K_1$ if and only if
$$\iota(E(\alpha(a^*a)))p = \iota(E(\alpha(a^*a))).$$
That is, $\alpha(a) \in J_1$ if and only if $\Phi(a^*a)p = \Phi(a^*a)$.
By Theorem~\ref{PseudoCarLatIso},
$$J = \{a\in A: \Phi(a^*a)\in\alpha(K)^{\dperp_{D_1}}\}.$$
Therefore, $\alpha(a) \in J_1$ if and only if $a \in J$.
It follows that the map
\begin{align*}
    \overline{\alpha}\colon  & A/J \longrightarrow A_1/J_1 \\
    & a + J  \longmapsto \alpha(a) + J_1
\end{align*}
is a well-defined $*$-monomorphism.

Let
$$ N = \{ n +J \colon n \in N(A,D)\}.$$
Then $N$ is a generating $*$-semigroup of normalizers for $A/J$.
As $\alpha$ is a regular morphism, $\alpha(n) + J$ is a normalizer of $D_1/K_1$ for each $n \in N(A,D)$.
Thus, by \cite[Proposition~6.1.6]{Pit25}, $\overline{\alpha}$ is a regular expansion.

The inclusion $(A_1/J_1, D_1/K_1)$ is Cartan by \cite[Theorem~4.8]{BFPR24}.
By \cite[Lemma~4.7]{BFPR24}, the conditional expectation $F \colon A_1/J_1 \rightarrow D_1/K_1$ is given by
$$ F(a + J_1) = E(a) + K_1.$$
Thus, if $a_1,\ldots,a_n \in A$, then
\begin{align*}
 &\alpha(a_1)E(\alpha(a_2))\ldots E(\alpha(a_n)) + J_1\\
= &(\alpha(a_1) + J_1) (E(\alpha(a_2))+J_1) \ldots (E(\alpha(a_n)) + J_1)\\
=& (\overline{\alpha}(a_1+J))(F(\overline{\alpha}(a_2+J)) \ldots (\overline{\alpha}(a_n+J)).
\end{align*}
As $A_1 = \cstar(\alpha(A) \cup E(\alpha(A)))$ it follows that
$$A_1/J_1 = \cstar(\overline{\alpha}(A/J) \cup F(\overline{\alpha}(A))).$$
Similarly
$$D_1/K_1 = \cstar(F(\overline{\alpha}(A/J))).$$
Thus $(A_1/J_1, D_1/K_1, \overline{\alpha})$ satisfies condition (1) and (2) of Definition~\ref{def: Cartan env}.

The inclusion $(D_1, \alpha(D))$ has the ideal intersection property since $(A_1,D_1,\alpha)$ is a Cartan envelope of $(A,D)$.
Therefore $(D_1/K_1, \alpha(D)/(J_1\cap \alpha(D))$ has the \iip\ by
Corollary~\ref{cor: abelian quotients} and
\cite[Proposition~2.2.4]{Pit25}.  Hence condition (3) of Definition~\ref{def: Cartan env} is satisfied.
Therefore $(A_1/J_1, D_1/K_1, \overline{\alpha})$ is a Cartan envelope for $(A/J, D/K)$.
\end{proof}

\section{Reduced Crossed Products and Virtual Cartan Inclusions}
\label{Sec:RCP}
\numberwithin{equation}{section}

The purpose of this section is to show that a large class of virtual
Cartan, but not Cartan, inclusions naturally arise in the context of
reduced crossed products.

Let $X$ be a compact Hausdorff space and consider the reduced crossed
product, $C(X)\rtimes_r\Gamma$ by a
discrete group $\Gamma$.  If the relative commutant of $C(X)$ in
$C(X)\rtimes_r\Gamma$ is abelian,~\cite[Theorem~6.14]{Pit17} shows
that $C(X)^c\subseteq C(X)\rtimes_r\Gamma$ is a virtual Cartan
inclusion.   In this section, we give a dynamical characterization of
when this inclusion is a Cartan inclusion under the assumption that
$X$ is
connected; this is accomplished in~Theorem~\ref{connect}.  Our
characterization produces a large class of virtual
Cartan inclusions which are not Cartan inclusions.

 \begin{remark} Under the assumption that $C(X)\rtimes_r\Gamma$ is
amenable, several of the results in this section can be obtained using
the groupoid literature.  In particular, in the amenable context,
Theorem~\ref{connect} and Proposition~\ref{noCEfam} below are consequences
of~\cite[Corollary~4.5]{BrownNagyReznikoffSimsWilliamsCaSuC*AlHaEtGr}
and~\cite[Proposition~4.1]{BrownNagyReznikoffSimsWilliamsCaSuC*AlHaEtGr}
respectively.   Since our results do not require an amenability
assumption and the theory of reduced crossed products is likely to be more
familiar to a broader audience, we have
chosen not to quote the groupoid literature here.
\end{remark}

The material in this section relies
upon~\cite[Section~6]{Pit17}, so we begin
with a summary of the
notation and results we require.

\medskip
\noindent\bf Notation
    and Background Results Used Throughout Section~\ref{Sec:RCP}:
    \rm
    \begin{notation}
\begin{enumerate}
 \item $X$ is a compact Hausdorff space.
 \item $\Gamma$ is a discrete group with unit $e$ acting via homeomorphisms on $X$.  For $s\in \Gamma$ and $x\in X$, $sx$ represents the image of $x$ under the homeomorphism corresponding to $s$.
 \item For $s\in \Gamma$ and $f\in C(X)$, $\tau_s(f)=f(s^{-1}x)$ is
   the corresponding automorphism of $C(X)$.  We will frequently write
   \[\D:=C(X),\] and will use $\D$ and $C(X)$ interchangeably.
   \item Let $C_c(\Gamma,\D)$ be the set of all functions
$a:\Gamma\rightarrow \D$ such that $\{s\in\Gamma: a(s)\neq 0\}$ is a
finite set.
Then $C_c(\Gamma,\D)$ is a $*$-algebra under the usual
convolution product and adjoint operation: for $a,b\in C_c(\Gamma,\D)$, 
\begin{equation}\label{convalg}
(ab)(t)=\sum_{r\in\Gamma} a(r)\tau_r(b(r^{-1}t))\dstext{and}
(a^*)(t)=\tau_t(a(t^{-1}))^*.
\end{equation}
Let \[\C=C(X)\rtimes_{r}\Gamma=\D\rtimes_r\Gamma\] be the reduced crossed product of
$C(X)$ by $\Gamma$, that is, the completion of $C_c(\Gamma,C(X))$ with
respect to the reduced \cstar-norm $\norm{\cdot}_r$.
\item The group $\Gamma$ is naturally embedded into $\C$ via $s\mapsto
\emb_s$, where $\emb_s$ is the element of $C_c(\Gamma,\D)$ given by
\[\emb_s(t)=\begin{cases} 0& \text{if $t\neq s$}\\ I&\text{if
    $t=s$.}\end{cases}\] 
Also, $C(X)$ is embedded into $C_c(\Gamma,\D)$ via the map $d\mapsto
d\emb_e$ and we  identify $C(X)$ with its image under this map.  

Notice that for $s\in \Gamma$ and $d\in \D$, 
$\emb_s d \emb_{s^{-1}}=\tau_s(d)$ and $\spn\{d\emb_s: d\in
\D, s\in\Gamma\}$ is norm dense in $\C$, so 
$\{\emb_s:s\in\Gamma\}\subseteq \N(\C,\D)$.  Thus $(\C,\D)$ is a
regular inclusion.

\item The map $\coexp:C_c(\Gamma,
\D)\rightarrow \D$ given by $\coexp(a)=a(e)$ extends to a faithful
conditional expectation $\coexp$ of $\C$ onto $\D$.  Likewise, the maps
$\coexp_s:C_c(\Gamma,\D)\rightarrow \D$ given by $\coexp_s(a)=a(s)$
extend to norm-one linear mappings $\coexp_s$ 
of $\C$ onto $\D$.  Notice that for
$a\in\C$ and $s\in\Gamma$,
\begin{equation}\label{Esdef}
  \coexp_s(a)=\coexp(a\emb_{s^{-1}}).
\end{equation}
If $a\in\C$ and $\coexp_s(a)=0$ for every $s\in \Gamma$, then $a=0$.  (See~\cite[Proposition~6.5]{Pit17} for details.)
\end{enumerate}
\end{notation}

 The following are from
  \cite[Definition~6.6]{Pit17}. \begin{definition}\label{variousdefs}
    \begin{enumerate} 
\item[(i)] For $s\in \Gamma$, let $\fix s=\{x\in X: sx=x\}$ be
  the set of fixed points of $s$.  
\item[(ii)] For $s\in\Gamma$, let
  $\fF_s=\{f\in\D: \supp(f)\subseteq  \intfix s\}$.
  Thus $\{\fF_s:s\in\Gamma\}$ is a family of closed  ideals in $\D$
  and $\fF_s\simeq C_0(\intfix s)$.  

\item[(iii)]
For $x\in X$, let  
$H^x:=\{s\in \Gamma: x \in \intfix s\}.$  Then $H^x$ is a normal
subgroup of $\Gamma$, see~\cite[Remark~6.7]{Pit17}.
(In~\cite{Pit17}, $H^x$ is called the \textit{germ isotropy group at $x$.}) 
\end{enumerate}
\renewcommand{\theenumi}{\alph{enumi}}
\end{definition}
Here is a description of the relative commutant of $\D$ in $\C$,
  denoted $\D^c$: 
  \begin{proposition}[{\cite[Proposition~6.8]{Pit17}}]\label{Dcdes}
      \begin{align*}
\D^c&=\{a\in\C: \tau_s(d)\coexp_s(a)=d\coexp_s(a) \text{ for all $d\in\D$ and all
 $s\in\Gamma$}\}\\
&=\{a\in\C: \coexp_s(a)\in\fF_s \text{ for all $s\in\Gamma$}\}.
                                                            \end{align*}
                                                          \end{proposition}

A useful faithful representation $\theta:
  \C\rightarrow \bh$, called the very discrete representation, is
  found in~\cite[Section 6]{Pit17}.   Using this representation,
 for each $x\in X$, one defines an isometry
 $V_x:\ell^2(H^x)\rightarrow \H$.  Here is a key fact about the isometries
 $V_x$ and the very discrete representation $\theta$.
\begin{proposition}[{\cite[Proposition~6.9]{Pit17}}]\label{evaluate} For $x\in X$ and $a\in \C$, define
$\Phi_x(a):=V_x^*\theta(a)V_x.$ Then $\Phi_x$ is a completely
positive unital mapping of $\C$ onto $\cstarred(H^x)$ and   
  $\Phi_x|_{\D^c}$ is a $*$-epimorphism of $\D^c$ onto $\cstarred(H^x)$.
\end{proposition}

\begin{flexstate}{Formula}{{\cite[Equation~6.6 on p.\ 403]{Pit17}}}\label{genD}
 Let $\lambda: \Gamma\rightarrow \ell^2(\Gamma)$ be the left regular
 representation.  (Following tradition, we write $\lambda_r$ instead of
 $\lambda(r)$.) For $d\in \D$ and $r\in \Gamma$, \begin{equation*}
\Phi_x(d\emb_r)=\begin{cases} 0&\text{if $r\notin H^x$,}\\
d(x)\lambda_r|_{\ell^2(H^x)}&\text{if $r\in H^x$.}
\end{cases}
\end{equation*}
\end{flexstate}

We will require the following results.

\begin{theorem}[{\cite[Theorem~6.11]{Pit17}}] \label{comdes}
  $\D^c$ is abelian if and only if $H^x$ is an abelian group for every
  $x\in X$.
\end{theorem}
\begin{theorem}[{\cite[Theorem~6.14]{Pit17}}]\label{arc->vc} If
  $\D^c$ is abelian, then the  inclusion $(\C,\D^c)$ is a virtual
    Cartan inclusion.
  \end{theorem}

We now turn to the goal of this section:  establishing Theorem~\ref{connect}.  The following result is inspired by \cite[Example~1.2]{Pit17}: it  gives concrete examples where there is no conditional expectation
of $\C$ onto $\D^c$.

\begin{proposition}\label{noCEfam}  Suppose $X$ is connected, $\D^c$
  is abelian, and there is $s\in\Gamma$
  such that \[\emptyset \neq (\fix s)^\circ
    \neq X.\]  
  Then there is no conditional expectation $\Delta:\C\rightarrow
  \D^c$.
\end{proposition}
\begin{proof}
Arguing by contradiction, suppose  $\Delta:\C\rightarrow \D^c$ is a
conditional expectation.
  Fix such an $s\in \Gamma$ satisfying the hypotheses.  Then 
  \[\{0\}\neq \fF_s
    \neq C(X).\]

Note  that for any $h\in C(X)$,
\[w_s^*\Delta(w_s)h=w_s^*\Delta(w_sh)=w_s^*\Delta(\tau_s(h)w_s)=w_s^*\tau_s(h)\Delta(w_s)=hw_s^*\Delta(w_s), \]
so $w_s^*\Delta(w_s)\in \D^c$. Likewise $\Delta(w_s)w_s^*\in \D^c$.  Therefore,
\begin{equation}\label{noCE1}
  w_s^*\Delta(w_s)=\Delta(w_s^*\Delta(w_s))=\Delta(w_s^*)\Delta(w_s)=\Delta(w_s)\Delta(w_s)^*=\Delta(w_s)w_s^*.
\end{equation}

Recalling the definition of $\bbE_s$ from~\eqref{Esdef}, put
\[f_s:=\bbE_s(\Delta(w_s))=\bbE(\Delta(w_s)w_s^*)=\bbE(w_s^*\Delta(w_s)).\]  Then
$f_s\in \fF_s$ and $f_sw_s\in \D^c$.
Since $f_s\in \fF_s$, $\tau_s(f_s)=f_s$.  Thus
$f_sw_s^*=w_s^*f_s\in \D^c$.  Then
\[f_sw_s^*\Delta(w_s)=\Delta(f_sw_s^*w_s)=f_s.\]
Hence
\[f_s=\bbE(f_sw_s^*\Delta(w_s))=f_s\bbE(w_s^*\Delta(w_s))=f_s^2.\]
 Connectivity of $X$
forces $f_s=0$ or $f_s=1$.    Since $f_s\in \fF_s\neq C(X)$, we get
$f_s=0$.
But then
\[0=f_s=\bbE(w_s^*\Delta(w_s))=\bbE(\Delta(w_s)^*\Delta(w_s)),\] whence
$\Delta(w_s)=0$ by faithfulness of $\bbE$.

Since $\fF_s\neq \{0\}$ we may choose $0\neq g\in\fF_s$.   Note that
$gw_s\in\D^c$ and 
$gw_s\neq 0$ because $(g
w_s)(gw_s)^*=gg^*\neq 0$.   Hence
\[0=g\Delta(w_s)=\Delta(gw_s)=gw_s\neq 0, \] an impossibility.
Therefore, there is no conditional expectation of $\C$ onto $\D^c$.
\end{proof}

The next lemma follows from a result of Choda.
\begin{lemma}[c.f. {\cite[Proposition~1]{ChodaCoBeSuSuDiC*CrPr}}]\label{Esubgp}  If $N$ is a subgroup of $\Gamma$, then
  $C(X)\rtimes_r N\subseteq C(X)\rtimes_r\Gamma$ and there is a
  conditional expectation
  $\Delta: C(X)\rtimes_r\Gamma\rightarrow C(X)\rtimes_r N$ which
  extends the map
  $C_c(\Gamma, C(X))\ni f\mapsto \begin{cases} 0& \text{ if } s\notin
    N\\ f(s) & \text{ if } s\in N.
  \end{cases}$
\end{lemma}

The following can be interpreted as providing examples of non-standard
Cartan incluions, along the lines
of~\cite{DuwenigGillaspyNortonReznikoffWrightCaSuNoPrTwGr}.

\begin{theorem}\label{connect}  Suppose $X$ is a compact, connected
  Hausdorff space, and the discrete group $\Gamma$ acts on $X$ by
  homeomorphisms.  Put $\C=C(X)\rtimes_r \Gamma$ and $\D=C(X)$.
  The
  following statements are equivalent.
  \begin{enumerate}
  \item\label{connect1} $(\C,\D^c)$ is a Cartan inclusion.
\item \label{connect0}  There is an abelian subgroup
  $N\subseteq \Gamma$ such that for  $s\in \Gamma$,
 \[(\fix s)^\circ =\begin{cases} X &\text{ if } s\in N\\ \emptyset
     &\text{ if } s\notin N.\end{cases}\]
\item \label{connect3}  There is an abelian subgroup $N$ of $\Gamma$ such that
  the action of $N$ on $X$ is trivial and $\D^c=C(X)\rtimes_r N\simeq
  C(X)\otimes \cstar(N)$.
\item\label{connect2}  $\D^c$ is abelian and there is a 
   conditional expectation $\Delta: \C\rightarrow \D^c$ (the
   conditional expectation is not assumed faithful).
 
\end{enumerate}
When this occurs, $N$ is a normal subgroup of $\Gamma$.
\end{theorem}
\begin{proof}  
 \eqref{connect1}$\implies$\eqref{connect0}   By hypothesis, $\D^c$ is abelian and  there exists a faithful conditional
 expectation $\Delta:\C\rightarrow \D^c$.
 Proposition~\ref{noCEfam} shows that for
  every $s\in\Gamma$,
\[(\fix s)^\circ\in \{\emptyset, X\}.\]

Let \[N:=\{s\in \Gamma: (\fix s)^\circ=X\}.\] Then $N$ is a subgroup
of $\Gamma$, and for each $x\in X$, $H^x=N$.  By
\cite[Remark~6.7]{Pit17}, $H^x$ is a normal subgroup of
$\Gamma$ for every $x$, hence $N$ is a normal subgroup.  Also since
$\D^c$ is abelian, $H^x$ is an abelian group for every $x$, whence $N$
is abelian.

\eqref{connect0}$\implies$\eqref{connect3}  Let $N$ be the subgroup
with the properties given in part~\eqref{connect0}.
Observe that for each $x\in X$, $H^x=N$.   Therefore $\D^c$ is
abelian by Theorem~\ref{comdes}.   Since \[\{d \omega_s: d\in \D \text{
  and } s\in N\}\subseteq \D^c,\] we obtain \[C(X)\rtimes_r N\subseteq
\D^c.\]   By hypothesis, the action of $N$ on $X$ is trivial.
Apply Lemma~\ref{Esubgp} to obtain a conditional expectation $\Delta:
\C\rightarrow C(X)\rtimes_r N$.

To show $\D^c=C(X)\rtimes_r N$, it suffices to show $\Delta$ maps onto
$\D^c$.  By assumption on $\Gamma$, for $x\in X$,
$\{s\in \Gamma: x\in \intfix s\}=N$, that is, for every $x\in X$,
\[H^x=N.\]
By the definition of $\Delta$ and Formula~\ref{genD}, for every $f\in
C(X)$ and $s\in \Gamma$,
\[\Phi_x(f\omega_s)=\Phi_x(\Delta((f\omega_s)).\]   Since
$\{f\omega_s: f\in C(X), \; s\in \Gamma\}$ has dense span in $\C$, we
conclude that for every $x\in X$,
\[\Phi_x=\Phi_x\circ\Delta.\]

 Let $a\in \D^c$.   By ~\cite[Corollary~6.10]{Pit17}, the map $\Phi:
 \D^c\rightarrow \bigoplus_{x\in X} \cstarred(H^x)$ given by $\Phi(h)(x)=\Phi_x(h)$
 is a $*$-monomorphism.   Since
 \[\Phi(\Delta(a))(x)=\Phi_x(\Delta(a))=\Phi_x(a)=\Phi(a)(x),\]  we obtain
 $\Delta(a)=a$.  Hence  $\Delta$ maps onto $\D^c$, so
 $\D^c=C(X)\rtimes_r N$.   That $C(X)\rtimes_r N\simeq
  C(X)\otimes \cstar(N)$ is because the action of $N$ on $X$ is trivial.

\eqref{connect3}$\implies$\eqref{connect2}
Since $N$ acts trivially on $X$, $\D^c=C(X)\rtimes_rN$ is abelian.
Applying Lemma~\ref{Esubgp}, we obtain a conditional expectation of
$\C$ onto $\D^c$.  This gives~\eqref{connect2}.

 \eqref{connect2}$\implies$\eqref{connect1}
 By~\cite[Theorem~6.15]{Pit17}, $(\C,\D^c)$ is a virtual Cartan
 inclusion.   Applying~\cite[Proposition~2.3.19]{Pit25} we see
 that $(\C,\D^c)$ is a Cartan inclusion.

 Finally, recall that during the proof of
 \eqref{connect1}$\implies$\eqref{connect0} we showed  $N$ is a normal
 subgroup. 
\end{proof}

\begin{flexstate}{Remark}{}\rm
  If $N=\{e\}$ in Theorem~\ref{connect}, we obtain an alternate proof
  (in the case that $X$ is connected) of the well-known
  characterization that $C(X)$ is a Cartan MASA in
  $C(X)\rtimes_r\Gamma$ if and only if the action of $\Gamma$ on $X$
  is topologically free.
    \end{flexstate}

    \begin{flexstate}{Remark}{}\rm
Theorem~\ref{connect} (or Proposition~\ref{noCEfam}) can be used to give  a wide variety of
  examples of virtual Cartan inclusions which are not Cartan
  inclusions.  For instance, when an abelian group $\Gamma$ acts on
  the connected and compact Hausdorff space $X$, $\D^c$ is necessarily abelian, and
  it is easy to find actions for which there exists $s\in \Gamma$ with
  $\emptyset\neq (\fix s)^\circ \neq X$ (i.e.
  \eqref{connect0} fails).  In this case, $(\C,\D^c)$ is virtual
  Cartan but not Cartan.
\end{flexstate}
\begin{flexstate}{Remark}{} \rm Suppose $\Gamma$ is a simple group, $X$ is connected
  and 
  the conditions of  Theorem~\ref{connect}
  hold.   Then either: $N=\{e\}$ or $N=\Gamma$.
  In the first case, the action of $\Gamma$
  on $X$ is topologically free so that $\D^c=\D$ and $(\C,\D)$ is a
  Cartan inclusion.  On the other hand, if 
  $N=\Gamma$, then $\Gamma$ is abelian so $\Gamma=\{e\}$;  thus
  $\C=\D$ in the second case.  (We thank Mikkel Munkholm for this observation.)
\end{flexstate}

\section{Examples}\label{Sec:Examples}

This section presents examples which illustrate the results of the
present paper as well as its predecessor~\cite{BFPR24}.

\numberwithin{equation}{subsection}

\subsection{Examples from Reduced Crossed Products}

 \begin{example}\label{nexttry} \rm
  We wish to illustrate that when the connectedness hypothesis found
   in Theorem~\ref{connect} fails, it is possible for $C(X)^c\subseteq
   C(X)\rtimes_r\Gamma$ to be a  Cartan subalgebra, yet
   there is $s\in \Gamma$ such that
   $\emptyset \neq (\fix s)^\circ\neq X$.  The
   Cartan inclusion exhibited here is the Cartan envelope of the
   virtual Cartan inclusion found in 
   Example~\ref{nexttryA}
   below.   As noted in Remark~\ref{nexttryA.0}, these examples
are related to the example found in \cite[Section 2]{PittsCoStReInII}.

 For  a self-adjoint
   unitary $I\neq U\in M_2(\bbC)$, let
   $P_\pm(U)=(I\pm U)/2$ be the spectral projections for $U$
   corresponding to the eigenvalues $\pm 1$.
 We shall fix the following two self-adjoint unitaries:  \[S=\begin{bmatrix}
     -1&0\\0 &1\end{bmatrix}\dstext{and}
   T=\begin{bmatrix}0&1\\1&0\end{bmatrix}.\]  They satisfy $ST=-TS$.

 Let $X$ be a compact Hausdorff space 
 and let \[\aug X:=  X\times \3.\]
 Regard
$C(\aug X)$ as the \cstaralg\
of all ordered triples $(f_1,f_2,f_3)$ of elements of $C(X)$, with
componentwise operations.
   
  Let $\xi:=(1 \phantom{ 1}  2)$ be the permutation of $\3$ which leaves $3$ fixed
  and exchanges $1$ and $2$.    Consider the homeomorphism $\tau$ of 
  $\aug X$  given by
  \[\tau(x,j)=(x,\xi(j)).\] 
  This is an action of the two element group $\Gamma=\{0,1\}$ on
  $\aug X$.  The corresponding automorphism of $C(\aug X)$ is
  $(f_1,f_2,f_3)\mapsto (f_2,f_1,f_3)$.

  We will often identify $C(X)\otimes M_2(\bbC)$ with $C(X,M_2(\bbC))$
  or $M_2(C(X))$
  without comment.    

  For $f\in C(X)$ and $M\in M_2(\bbC)$, we will use the notation $fM$
for the function in $C(X,M_2(\bbC))$ given by $x\mapsto f(x)M$. 

  Define $\pi: C(\aug X)\rightarrow C(X,M_2(\bbC))\oplus C(X,M_2(\bbC))$ by
  \[\pi(f_1,f_2,f_3)=(f_1P_-(S) + f_2 P_+(S))\oplus f_3 I_2 \text{ (in
      matrix form this equals
      $
      \left[\begin{smallmatrix}f_1&0\\0&f_2\end{smallmatrix}\right]\oplus
      \left[\begin{smallmatrix} f_3&0\\
          0&f_3\end{smallmatrix}\right]$)}.\] Writing $I_X$ and $I_2$
  for the identity elements of $C(X)$ and $M_2(\bbC)$ respectively, let \[\tilde
  T:=I_X T, \text{ (in matrix form this is $\left[\begin{smallmatrix}0&I_X\\
        I_X&0\end{smallmatrix}\right]$),}\]
and let $\A$ be the \cstaralg\ generated by
  $\pi(C(\aug X))$ and $\tilde T\oplus \tilde T$:
\[\A=\cstar(\pi(C(\aug X)) \cup \{\tilde T\oplus
  \tilde T\}).\]

Let 
\[ C(X,\cstar(T)):=\left\{gP_-(T)+hP_+(T): g, h\in
    C(X)\right\}=\left\{\frac{1}{2}\left[\begin{matrix}
        g+h&g-h\\g-h&g+h\end{matrix}\right]: g, h\in C(X)\right\}
    ;\] likewise we write  $C(X, \cstar(S))$ for
$\left \{\begin{bmatrix}g &0\\0&h\end{bmatrix} : g,h\in C(X)\right\}$.  
Thus,
\[\A=C(X,M_2(\bbC))\oplus C(X, \cstar(T)).\]

It is not difficult to show that
the map
\begin{equation}\label{nexttry1}
  C_c(\Gamma,
C(\aug X))\ni a\mapsto \pi(a(0))I_\A+\pi(a(1))(\tilde T\oplus \tilde T)\in
\A
\end{equation}
determines 
an isomorphism of $C(\aug X)\rtimes_r \Gamma$ onto $\A$. 
Using this isomorphism, we
identify $C(\aug X)\rtimes_r \Gamma$ with $\A$. 

By Proposition~\ref{Dcdes} (Definition~\ref{variousdefs} explains the
notation used in that result), the relative commutant of
$C(\aug X)$ in $\A$ is
   \begin{equation}\label{nexttry2}\B:=C(\aug X)^c=C(X, \cstar(S))\oplus
     C(X,\cstar(T))).
   \end{equation} 
   There is a faithful conditional expectation of $\A$ onto $\B$ given by
\begin{equation}\label{nexttry3}\Delta\left(\begin{bmatrix}f_{11}&f_{12}\\f_{21}&f_{22}\end{bmatrix}\oplus \begin{bmatrix}g&h\\h&g\end{bmatrix}\right)
= \begin{bmatrix}f_{11}&0\\0&f_{22}\end{bmatrix}\oplus \begin{bmatrix}g&h\\h&g\end{bmatrix}.
\end{equation}
We conclude that 
\[(\A,\B)=(C(\aug X)\rtimes_r\Gamma, C(\aug X)^c)\] is a Cartan inclusion.
\end{example}

\begin{example}\label{nexttryA}  \rm We continue with the notation and
  inclusions $(\A,\B)$ and $(\A, C(\aug X))$ found in Example~\ref{nexttry}, but make the additional
  assumption that $X$ is a connected set having more than one
element.  Fix
  $p\in X$.  Note that by connectedness, $p$ is an accumulation point
  of $X$.

Identify the points in the set $\{p\}\times \3\subseteq \aug X$ to obtain the
quotient space $Y:=\aug X/\sim$.     Then $Y$ is a compact, connected
Hausdorff space.   We will use $\phi$ for the quotient
mapping of $\aug X$ onto $Y$.     We write $\dot p$ for the common
image of $(p, i)$ in $Y$ ($i=1,2,3)$.

This gives an embedding $\alpha_0:C(Y)\rightarrow C(\aug X)$ given by
\begin{equation}\label{nexttryA0}\alpha_0(f)= f\circ \phi.
\end{equation}
Note that 
\[\alpha_0(C(Y))=\{f\in C(\aug X): f(p,1)=f(p,2)=f(p,3)\}.\]

The action of $\Gamma$ on $\aug X$ leaves $\{p\}\times\3$ invariant,
hence determines an action of $\Gamma$ on $Y$.  This action satisfies:
\[s\dot p=\dot p\dstext{and}  s\phi(x)=\phi(sx), \quad s\in \Gamma, \; x\in \aug X.\]
Let
\[\C:=C(Y)\rtimes_r \Gamma \dstext{and let} \D:=C(Y)\subseteq \C\] 

If
$s$ is the non-trivial homeomorphism of $Y$ arising from $\Gamma$,
then 
\[\emptyset\neq (\phi(X\times\{3\}))^\circ=(\fix s)^\circ\neq Y.\]  By
Theorem~\ref{comdes} and Theorem~\ref{arc->vc},
  $(\C, \D^c)$ is a virtual Cartan inclusion, and there is no conditional expectation of $\C$
  onto $\D^c$.   Therefore, $(\C,\D^c)$ is a virtual Cartan inclusion
but not a Cartan inclusion.

Extend the map $\alpha_0$ given in~\eqref{nexttryA0} to a map,  $\alpha:\C\rightarrow\A= C(\aug X)\rtimes_r \Gamma$ by
\[\alpha(a)= \pi(\alpha_0(a(0))) I_\A + \pi(\alpha_0(a(1)))(\tilde T\oplus
  \tilde T) \quad a\in C_c(\Gamma, C(Y)).\]

The image of $\alpha$ is
\begin{align}\label{nexttryA2}
&\left
  \{\begin{bmatrix}f_{11}&f_{12}\\f_{21}&f_{22}\end{bmatrix}\oplus \begin{bmatrix}g&h\\h&g\end{bmatrix}\in\A:
  \text{ all matrix entries belong to } \alpha_0(\D)
                                                                                          \right\}\\
  &\simeq \{f\in C(Y, M_2(\bbC))): f(\phi(x,3))\in \cstar(T) \text{ for all
    }x\in X\}.\label{nexttryA3}
\end{align}

We aim to show $(\A, \B, \alpha)$ is a Cartan
envelope for $(\C,\D^c)$, 

We have,
\[\alpha(\D^c)=\{f\in \B: f(p,1)=f(p,2)=f(p,3)\},\] from which it
follows 
that $\alpha(\D^c)\subseteq \B$ has the ideal intersection
property.   Since $(\C, \D^c)$ is a virtual Cartan
inclusion and $\alpha|_{\D^c}$ is one-to-one, we find $\alpha$ is a
$*$-monomorphism.

Note that~\eqref{nexttry3} gives
\[\Delta(\tilde T\oplus \tilde T)=0\oplus \tilde T,\] and routine
arguments (using
$\tilde T^2=\left[\begin{smallmatrix} I_X& 0\\
    0&I_X\end{smallmatrix}\right]$) show
that
\[\B=\cstar(\Delta(\alpha(\C)))\dstext{and}\A= \cstar(\B\cup
  \alpha(\C)).\] It remains to show $\alpha$ is a regular
homomorphism.  For this, we describe $N(\C,\D^c)$.

\bf Normalizers:  \rm   Suppose $a\in N(\C,\D^c)$, where we view $\C$
and $\D^c$ as subalgebras of $C(Y,M_2(\bbC))$.  Then for every $y\in
Y$,  $a(y)$ must
normalize the fiber for $\D^c$ over $y\in Y$.  Since $\cstar(S)\cap
\cstar(T)=\bbC I$, and $N(M_2(\bbC),\bbC)$
consists of scalar multiples of unitaries in $M_2(\bbC)$, we obtain
for $n\in\3$ and $x\in X$,
\[a(\phi(x,n) )\in \begin{cases} N(M_2(\bbC), \cstar(S))& n\in \{1,2\}\\
    \bbC \U(M_2(\bbC)) & \phi(p,n)=\dot p\\ \cstar(T) & n=3.
  \end{cases}
  \]  By continuity of $a$, for $n\in \{1,2\}$,
  \[\lim_{\phi(x.n)\rightarrow \dot p}a((\phi(x,n))=a(\dot p),\] so  $a(\dot p)\in
  \bbC\U(\cstar(S))$.  Similarly, $\lim_{\phi(x,3)\rightarrow \dot p}
  a(\phi(x,3))=a(\dot p)\in \bbC \U(\cstar(T))$.  This forces  $a(\dot
  p)\in \bbC I$.   Thus,
  \begin{equation}\label{nexttryA1} \begin{split}N(\C, \D^c)&=
\left\{a\in \C: a(\phi(x,n))\in N(M_2(\bbC), \cstar(S)) \text{ for }
        n\in \{1,2\},\right. \\ &\qquad\left.  a(\dot p)\in \bbC I, \text{ and } a(\phi(x,3))\in \cstar(T)
     \right\}.\end{split}
    \end{equation}  

It is now easy to see that
$\alpha(N(\C,\D^c)\subseteq N(\A,\B)$, so $\alpha$ is a regular
map.   Therefore $(\A,\B,\alpha)$ is a Cartan envelope for $(\C,\D^c))$.
    
   \textbf{The Pseudo-Expectation.}  Since $\alpha(\D^c)\subseteq \B$
   has the ideal intersection property, it follows that this inclusion
   has the faithful unique pseudo-expectation property.  Let
   $(I(\D^c), \iota)$ be an injective envelope for $\D^c$ and let
   $\iota_1:\B\rightarrow I(\D^c)$ be the pseudo-expectation for
   $(\B, \alpha(\D^c))$.  Then $\iota_1$ is a $*$-monomorphism and
   (using the fact that $(\C,\D^c)$ has the faithful unique
   pseudo-expectation property) we see the pseudo-expectation for $(\C,\D^c)$ is
\begin{equation}\label{nexttryA4}\Phi:=\iota_1\circ\Delta\circ\alpha.
\end{equation}
By~\cite[Proposition~6.2]{Pit21}, $\Phi$ is an invariant pseudo-expectation.

   \textbf{Regular Ideals.}
Using \eqref{nexttryA4}, we see that for $K\idealin \D^c$,
the definitions of $J_K$ and $L_K$
found in \cite[Notation~3.9]{BFPR24}
become
\[J_K=\{c\in \C: \Delta(\alpha(c^*c))\in \alpha(K)\}\dstext{and} L_K=\{c\in
  \C: \Delta(\alpha(c^*c))\in \alpha(K)^{\dperp_\B}\}.\]

Let $3^-$ and $3^+$ be symbols intended to represent a ``splitting''
of the
element $3\in \{1,2,3\}$ into two separate elements; we shall consider
the four element set $\{1,2,3^-,3^+\}$.

The maximal ideal spaces of $\B$ and $\D^c$ are
$X\times \{1,2,3^-,3^+\}$ and $(X\times \{1,2,3^-,3^+\})/\sim$ respectively, where
$\sim$ now means that 
the points $(p,1), (p,2), (p,3^-), (p,3^+)$ in $\hat\B$ have been identified.
Let \[\phi': X\times \{1,2,3^-,3^+\}\rightarrow (X\times \{1,2,3^-,3^+\})/\sim\]
be the quotient map and let $\dot p'= \phi'(p,k)$,
$k\in \{1,2,3^-,3^+\}$.

Let $K\idealin \D^c$ be a regular ideal and let $\Omega=\supp(K)$, that is,
the set of all $(x,k)\in X\times \{1,2,3^-,3^+\}$ for which there is $a\in K$
with $a(\phi'(x,k)))\neq 0$.   The support of $\alpha(K)^{\dperp_\B}$ is
$\left(\overline{\phi'{}^{-1}(\Omega)}\right)^\circ$, so that
\[\alpha(K)^\dperp=\left\{h\in C(X\times \{1,2,3^-,3^+\}): \supp (h) \subseteq
    \left(\overline{\phi'{}^{-1}(\Omega)}\right)^\circ\right\}.\]
Notice that  $K$ is an invariant regular ideal if and only if the
following condition holds:  for each $x\in X$, $\{\phi'(x,1),
\phi'(x,2)\}\subseteq \supp(K)$ or $\{\phi'(x,1),
\phi'(x,2)\}\cap \supp(K)=\emptyset$. 
By 
Theorem~\ref{PseudoCarLatIso}, the Boolean algebra of regular ideals
of $\C$ is isomorphic to the Boolean algebra of 
regular invariant ideals of $\D^c$.
\end{example}

\begin{remark} It is worth noting that the relationship between the
  Cartan inclusion $(\A,\B)$ of Example~\ref{nexttry} and the virtual
  Cartan inclusion $(\C,\D)$ of Example~\ref{nexttry2} is a particular
  case of the far more general setting found in
  \cite[Proposition~5.31]{Pit21}.   That proposition gives a method
  for 
  constructing a virtual Cartan inclusion $(\C,\D)$ from  a Cartan
  inclusion $(\A,\B)$ and a subalgebra $\D\subseteq \B$ having the
  \iip; in some cases $(\A,\B)$ is the Cartan envelope for $(\C,\D)$.
\end{remark}

\subsection{Rational Rotation Algebras}
\begin{example}\label{rotalg} \rm
  In this example, we take $X=\bbT$, $\Gamma=\bbZ$, and let
  $\lambda=\exp(2\pi i p/q)$, where $\gcd(p,q)=1$ and
  $1\leq p\leq q-1$.  The action of $\bbZ$ on $\bbT$ and the
  corresponding action of $\bbZ$ on $C(\bbT)$ are given by
\[nx:=\lambda^n x \dstext{and} \alpha_n(f)(x)=f(\overline\lambda^nx),
  \quad  n\in \bbZ, x\in \bbT.\]    Put \[\C:=C(\bbT)\rtimes_r \bbZ,
  \quad\D:=C(\bbT)\subseteq \C.\]    Let $V\in C(\bbT)$ be
the function $V(x)=x$ and denote by $U=w_1$, (so the canonical copy
of $\bbZ$  in $\C$ is $\{U^n: n\in \bbZ\}$).    Since
$UVU^*=\alpha_1(V)=\overline\lambda V$, we have
\[VU=\lambda UV.\]  Noting that 
\[\fix(k)=\begin{cases} \bbT & k\in q\bbZ\\ \emptyset & k\notin
    q\bbZ,\end{cases}\] Theorem~\ref{connect} shows 
$(\C,\D^c)$ is a Cartan inclusion.   Our aim is to determine the
regular ideals in $\C$.

First, it follows from Proposition~\ref{Dcdes} that $\D^c$ is the
\cstaralg\ generated by $\{V, U^q\}$.  Thinking of $C_c(q\bbZ,
C(\bbT))\subseteq C_c(\bbZ, C(\bbZ))$, we find $C(\bbT)\rtimes_r(q\bbZ)$ embeds
isometrically into $C(\bbT)\rtimes_r \bbZ$.
The action of $q\bbZ$ is trivial on $C(\bbT)$, so $C(\bbT)\rtimes_r
(q\bbZ)$ is isomorphic to $\cstar(V)\otimes \cstar(U^q)$.    We conclude 
\[\D^c=C(\bbT)\rtimes_r q\bbZ \text{ and } \D^c\simeq  C(\bbT^2).\]

Let $N$ be the group generated by $\{V, U\}$.   Then
$N\subseteq \N(\C, \D^c)$ is a
generating 
$*$-semigroup of normalizers.   For $j, k\in \bbZ$ and any finite sum
of the form $a=\sum_{s\in q\bbZ} f_s U^s\in \D^c$, we have
\[V^jU^ka U^{*k} V^{*j}=\sum_{s\in q\bbZ} \alpha_k(f_s)U^s;\] here we
have used the fact that $U^s$ commutes with $V$ for every $s\in
q\bbZ$.   Thus, viewing $\widehat{\D^c}=\bbT^2$, the
action of $N$ on $\widehat{\D^c}$  is given by
$\beta_{V^jU^k}(x,y)=(\lambda^j x, y)$,  that is,   the action
determined by $N$ is  ``rotation in the first variable'' and
the identity in the second.

By \cite[Corollary~3.20]{BFPR24}, the
collections of $N$-invariant regular ideals in $\D^c$ and invariant
regular ideals in $\D^c$ coincide.  Also, the regular invariant ideals
in $\D^c$ correspond to regular open sets $H\subseteq \bbT^2$
invariant under rotation in the first variable.

Finally,  \cite[Proposition~6.3]{BFPR24}
 gives a description of the regular ideals in the rational rotation
algebra $\C$ in terms of the regular invariant ideals in $\D^c$.

It is possible for two non-isomorphic \cstaralg s to
 have isomorphic lattices of regular ideals,
               see~\cite[Proposition~5.6]{PittsIrMaIsBoReOpSeReId} for
               examples of pairs of commutative \cstaralg s with this behavior.  We conclude this example
 by observing that such behavior occurs in the present context of
 rational rotation algebras. Suppose $q$ is fixed and for $k=1,2$,
 $1\leq p_k\leq q-1$ satisfy $\gcd(p_k, q)=1$.  Put
 $\lambda_k:= \exp(2\pi ip_k/q)$.  Consider the two actions of $\bbZ$
 on $\bbT$ corresponding to rotation by $\lambda_k$ and let
 $\C_1:=C(\bbT)\rtimes_{r,\lambda_1}\bbZ$ and
 $\C_2:=C(\bbT)\rtimes_{r,\lambda_2}\bbZ$ be the associated crossed
 products.  This gives two inclusions $(\C_k, \D_k)$ where
 $\D_k=C(\bbT)\subseteq \C_k$ and
 $\D_k^c\simeq C(\bbT)\otimes C(\bbT)$.  The lattices of regular
 invariant ideals for $\D^c_k$ are isomorphic because they arise from
 regular open sets in $\D^c$ invariant under rotation by
 $\exp(2\pi i/q)$ in the first variable. Thus the lattices of regular
 ideals in $\C_1$ and $\C_2$ are isomorphic.  However, if
 $p_1/q\notin\{p_2/q, (q-p_2)/q\}$, then $\C_1$ and $\C_2$
 are not isomorphic (see~\cite{YinHongShengSiPrClRaRoC*Al} for a short
 proof of this fact).
 \end{example}

 \subsection{An Example 
   Arising from Higher-Rank Graphs} \label{hrg}
 We use groupoids to describe \cstaralg s of higher rank graphs.
 A groupoid $G$ is a small category in which every element is
invertible.  We identify the identity morphisms with the objects and
refer to them collectively as the unit space $G^{(0)}$.  A topological
groupoid is a groupoid endowed with a locally compact Hausdorff
topology in which composition and inversion are continuous.  Our
interest is with \'etale groupoids (see \cite{simsnotes} for a brief
introduction to \'etale groupoids and their \ca-algebras).  If we let
$r_G(\gamma)$ be the range of an element $\gamma\in G$ and
$s_G(\gamma)$ be its source, the isotropy subgroupoid is defined to be
the set
\[\Iso(G)=\{\gamma\in G: r_G(\gamma)=s_G(\gamma)\}.\]
Theorem~3.1(b) of \cite{BrownNagyReznikoffSimsWilliamsCaSuC*AlHaEtGr}
implies that $\cstarred(\Iso(G)^\circ)$ has the ideal intersection
property in $\cstarred(G)$. Furthermore,
\cite[Theorem~4.3]{BrownNagyReznikoffSimsWilliamsCaSuC*AlHaEtGr} gives
conditions that guarantee that $\cstarred(\Iso(G)^\circ)$ is a MASA in
$\cstarred(G)$.  These conditions are always satisfied if $G$ is
constructed from a higher-rank graph (see \cite{KumPas00} for the
construction). Thus for groupoids constructed from higher-rank graphs,
the inclusion $(\cstarred(G),\cstarred(\Iso(G)^\circ))$ is always virtual
Cartan and hence pseudo-Cartan.   This inclusion is the same inclusion
as found in Example~\ref{ex: k-graph}. 

Example~4.7 from \cite{BrownNagyReznikoffSimsWilliamsCaSuC*AlHaEtGr}
illustrates (in the language of this paper) that pseudo-Cartan
inclusions that are not Cartan arise naturally  in the context of
higher-rank graph algebras.  We will further analyze the
pseudo-Cartan inclusion coming from this example, computing its
Cartan envelope and using the Cartan envelope, along with
Theorem~\ref{thm: all latice isomorphisms}, to describe the regular
ideals of the resulting higher-rank graph algebra.  The method used
here, of ``splitting'' the graph into two pieces, was inspired by Example~\ref{nexttry3}.

We first present a general result.

\begin{lemma}[{\cite[Proposition~3.3]{MuhlyRenaultWilliamsCoTrGrC*AlIII}}]\label{dual bundle}
Suppose $\mathcal{A}\stackrel{p}{\to} X$ is an \'etale bundle of abelian groups. 
Then $\caop(\mathcal{A})\cong C_0(\hat{\mathcal{A}})$, where $\hat{\mathcal{A}}=\{(\xi,x): x\in X, \, \xi \in \widehat{\mathcal{A}_x}\}$ and the topology on $ \hat{\mathcal{A}}$ is characterized by $(\xi_i,x_i)\to (\xi, x)$ if and only if \begin{itemize} \item $x_i\to x$ and \item for all nets $a_i\to a\in \mathcal{A}$ with $a_i\in p\inverse(x_i)$, we have $\xi_i(a_i)\to \xi(a)$. \end{itemize}
\end{lemma}
\begin{remark} Note that
 \cite[Proposition~9.81]{WilliamsToKiGrC*Al} shows that the bundle
 $\mathcal{A}$ from Lemma~\ref{dual bundle} is amenable, so
 $\caop(\mathcal{A})=\cstarred(\mathcal{A})$.  
\end{remark}

\begin{example}\label{2grExPre} Consider the 2-colored graph in Figure~\ref{2grEx}.

\begin{figure}[h!]
\begin{tikzpicture}[scale=1.5, >=stealth, decoration={markings, mark=at position 0.5 with {\arrow{>}}}]
  \node[circle, inner sep=1.5pt, fill=black, label={u}] (v) at (0,0) {};
  \node[circle, inner sep=1.5pt, fill=black, label={v}] (u) at (2,1) {};
  \node[circle, inner sep=1.5pt, fill=black, label={w}] (w) at (2,-1) {};
  \draw[blue, postaction=decorate] (v) .. controls +(0.75,0.75) and +(-0.75,0.75) .. (v) node[above, pos=0.5, black] {\small$e_b$};
  \draw[blue, postaction=decorate] (u) .. controls +(0.75,0.75) and +(-0.75,0.75) .. (u) node[above, pos=0.5, black] {\small$f_b$};
  \draw[blue, postaction=decorate] (w) .. controls +(0.7,0.7) and +(-0.65,0.65) .. (w) node[circle, inner sep=0.1pt, above, pos=0.5, black] {\small$g_b$};
  \draw[blue, postaction=decorate] (w) .. controls +(1,1) and +(-1,1) .. (w) node[circle, inner sep=0.1pt, above, pos=0.5, black] {\small$h_b$};
  \draw[red, dashed, postaction=decorate] (v) .. controls +(0.75,-0.75) and +(-0.75,-0.75) .. (v) node[below, pos=0.5, black] {\small$e_r$};
  \draw[red, dashed, postaction=decorate] (u) .. controls +(0.75,-0.75) and +(-0.75,-0.75) .. (u) node[below, pos=0.5, black] {\small$f_r$};
  \draw[red, dashed, postaction=decorate] (w) .. controls +(0.65,-0.65) and +(-0.65,-0.65) .. (w) node[circle, inner sep=0.1pt, below, pos=0.5, black] {\small$g_r$};
  \draw[red, dashed, postaction=decorate] (w) .. controls +(1,-1) and +(-1,-1) .. (w) node[circle, inner sep=0.1pt, below, pos=0.5, black] {\small$h_r$};
  \draw[blue, postaction=decorate, out=200, in=40] (u) to node[above, pos=0.5, black] {$\alpha_b$} (v);
  \draw[blue, postaction=decorate, out=140, in=340] (w) to node[above, pos=0.5, black] {$\beta_b$} (v);
  \draw[red, dashed, postaction=decorate, out=220, in=20] (u) to node[below, pos=0.5, black] {$\alpha_r$} (v);
  \draw[red, dashed, postaction=decorate, out=160, in=320] (w) to node[below, pos=0.5, black] {$\beta_r$} (v);
\end{tikzpicture}
\caption[b]{Skeleton of the 2-graph \cite[Example~4.7]{BrownNagyReznikoffSimsWilliamsCaSuC*AlHaEtGr}.} \label{2grEx}
\end{figure}

Define factorization rules by
\begin{gather}\label{factrule}
e_be_r=e_re_b, \ e_b \alpha_r = e_r \alpha_b,\ e_b\beta_r = e_r\beta_b,\
  \alpha_bf_r = \alpha_r f_b,\ f_b f_r = f_r f_b,\
  \beta_b g_r = \beta_r g_b,\\ \label{2grExa}
  \beta_b h_r = \beta_r h_b, \ g_bg_r = g_r g_b,\ g_b h_r = h_rg_b,\ h_bg_r = g_rh_b,\text{ and }
  h_b h_r = h_r h_b.\notag
\end{gather}

 \begin{dremark}{Notation}\label{lamdef}
For the remainder of this section, let $\Lambda$ be the  $2$-graph defined by the 2-colored graph in
Figure~\ref{2grEx} and the factorization rules~\eqref{factrule}, with    degree
map $d_\Lambda: \Lambda \to \bbZ^2$     satisfying $\{\alpha_b,
\beta_b, e_b, f_b, g_b, h_b\}\subseteq d^{-1}\{(1,0)\}$ and
$\{\alpha_r, \beta_r, e_r, f_r, g_r, h_r\}\subseteq d^{-1}\{(0,1)\}$ (see \cite{HRSW} for the
relationship between $k$-colored graphs  with factorization rules
 and $k$-graphs).    \end{dremark}

  We will consistently identify paths in $\Lambda$
with sequences of edges under the equivalence relations defined by the
factorization rules above; see Remarks~2.2 in \cite{KumPas00} and
Proposition~5.7 in \cite{HRSW}.

The associated path groupoid, $G_\Lambda$, is defined as follows,
where $\Lambda^\infty$ denotes the infinite path space; details can be
found in Section 2 of \cite{KumPas00}.
\begin{equation*} \begin{split} G_\Lambda&=\{(x,l,y)\in
    \Lambda^\infty\times \bbZ^2\times \Lambda^\infty:\exists \mu,\nu
    \in \Lambda, \, z\in \Lambda^\infty \, \text{ such that }  x=\mu z, \, y=\nu z, \, \text{and}\\
    & \phantom{=\{(x,l,y)\in \Lambda^\infty\times \bbZ^2\times
      \Lambda^\infty: \exists} d_\Lambda(\mu)-d_\Lambda(\nu)=l\}
    \\
    & = \{(x,l,y) \in \Lambda^\infty\times \bbZ^2 \times
    \Lambda^\infty : \exists p, q \in \mathbb{Z}^2 \, \, \text{ such
      that } \sigma^p(x)=\sigma^q(x) \text{ and }
    p-q=l\}, \end{split}
\end{equation*} where the shift map
$\sigma: \Lambda^\infty \rightarrow \Lambda^\infty$ is defined by
$\sigma^p(x)(m,n)=x(m+p_1, n+p_2)$ for
$p=(p_1, p_2) \in \mathbb{Z}^2$.

The units of $G_\Lambda$ are of the form $(x,0,x)$, which we often identify with $\Lambda^\infty$. The topology of $G_\Lambda$ is induced from a basis of cylinder sets
\[
Z_{G_\Lambda}(\mu,\nu)=\{(\mu x, \, d_\Lambda(\mu)-d_\Lambda(\nu), \, \nu x): x\in s(\mu)\Lambda^\infty\}
\]
defined by pairs $\mu,\nu\in \Lambda$ with $s(\mu)=s(\nu)$. In this topology, each $Z_{G_\Lambda}(\mu,\nu)$ is compact \cite[Proposition~2.8]{KumPas00}.

\begin{remark}
 By Theorem 5.5 of \cite{KumPas00}, $G_\Lambda$ is amenable so that $\caop(G_\Lambda) = \cstarred(G_\Lambda)$. We continue to write $\cstarred(G_\Lambda)$ throughout since most of the groupoid literature we cite concerns the reduced \ca-algebra.
\end{remark}

By \cite[Example~4.7]{BrownNagyReznikoffSimsWilliamsCaSuC*AlHaEtGr},
$\Iso(G_\Lambda)^\circ$ is not closed in $G_\Lambda$. While
$(\cstarred(G_\Lambda),\cstarred(\Iso(G_\Lambda)^\circ))$ is a regular
MASA inclusion
(\cite[Theorem~4.3]{BrownNagyReznikoffSimsWilliamsCaSuC*AlHaEtGr}),
$\cstarred(\Iso(G_\Lambda)^\circ)$ is not the image of a conditional
expectation
(\cite[Proposition~4.1]{BrownNagyReznikoffSimsWilliamsCaSuC*AlHaEtGr}).
We shall identify the Cartan envelope of
$(\cstarred(G_\Lambda),\cstarred(\Iso(G_\Lambda)^\circ))$ and use this
to describe the regular ideals of $\cstarred(G_\Lambda)$.

We start by establishing some notation and making some observations
about the infinite path space of  $\Lambda$.  We denote
\[ e^\infty =e_re_be_re_b\cdots \quad \text{and} \quad f^\infty = f_r f_b f_r f_b\cdots \] and partition the infinite path space $\Lambda^\infty$ into three disjoint sets,
\[ X_v=\Lambda v \Lambda^\infty=\Lambda f^\infty, \quad X_w=\Lambda w
  \Lambda^\infty, \quad \text{and} \quad X_u=\{e^\infty\}.\]For
$H \subseteq G_\Lambda$, and $X \subseteq \Lambda^\infty$, denote
$H|_X=XHX=\{(x,l,y) \in H \, : \, x, y \in X\}$.

Using the fact that $f^\infty = f_b^r f^\infty$ and the factorization
conditions, we find that for all $s,t \in \mathbb{N}$,
\begin{equation}\label{v reduction}
e_b^se_r^t \alpha_b f^\infty=e_b^{s+t}\alpha_r f^\infty=e_b^{s+t}\alpha_b f^\infty.
\end{equation}
Similar calculations elsewhere on the graph yield
\begin{eqnarray*}
  X_v& =& \{ e_b^n \alpha_b f^\infty: n\in \mathbf{N} \}\cup \{f^\infty\}  \\
  X_w &=
        & \{e_b^n \beta_b x: x\in w\Lambda^\infty, \, n\in \mathbf{N}\}\cup w\Lambda^\infty 
\end{eqnarray*}

The Cartan envelope will be described using the subgraphs $\Lambda_V$
and $\Lambda_W$ induced, respectively, by vertex sets
$\Lambda^0_V:=\{u,v\}$ and $\Lambda^0_W:=\{u,w\}$. Thus
$\Lambda^{\infty}_V=X_v \cup \{e^\infty\}$ and
$\Lambda^\infty_W=X_w \cup \{e^\infty\}$.

\begin{lemma}\label{iso Gv}
  Let $\Lambda_V$ be the subgraph of $\Lambda$ induced by the vertex set $\{u,v\}.$
  \begin{enumerate} \item For $k,\ell, k', \ell' \in \mathbb{N}$, the
    set
    \[ Z_{G_{\Lambda_V}}(e_b^k e_r^\ell, e_b^{k'} e_r^{\ell'})
      \subseteq \Iso(G_{\Lambda_V})\quad\text{ iff }\quad
      k-k'=-(\ell-\ell').\] Thus
    $\Iso(G_{\Lambda_V})^\circ|_{\{e^\infty\}}=\{(e^\infty, (m, -m),
    e^\infty): m\in \bbZ\} \cong \{e^\infty\} \times \mathbb{Z}$.
 \item
$\Iso(G_{\Lambda_V})^\circ|_{X_v}=\Iso(G_{\Lambda})^\circ|_{X_v}\cong
X_v\times \bbZ^2$.
 \end{enumerate}
\end{lemma}

\begin{proof} By definition, for any $k, k', \ell, \ell',$
  $Z_{G_{V}}(e_b^k e_r^\ell, e_b^{k'} e_r^{\ell'})|_{\{e^\infty\}}
  \subseteq \Iso(G_{\Lambda_V})$. Using (\ref{v reduction}), we have
\begin{align*} Z_{G_{\Lambda_{V}}}(e_b^k e_r^\ell, e_b^{k'} e_r^{\ell'})|_{X_v}&= \{(e_b^k e_r^\ell e_b^n \alpha_b f^\infty, (k-k', \ell-\ell'), e_b^{k'} e_r^{\ell'} e_b^n \alpha_b f^\infty): n\in \mathbf{N}\}\\
&= \{(e_b^{k +\ell+n}\alpha_b f^\infty, (k-k', \ell-\ell'), e_b^{k'+\ell'+n} \alpha_b f^\infty): n\in \mathbf{N}\}.
\end{align*} Noting that $e_b^{k +\ell+n}\alpha_b f^\infty=e_b^{k'+\ell'+n} \alpha_b f^\infty$  if and only if  $k-k'=-(\ell-\ell')$ gives the result.

For (b), by the definition of the topology on $G_\Lambda$, $X_v$ is
discrete. Thus every element of $\Iso(G_\Lambda)^\circ|_{X_v}$ is in
the interior.  Note that for each $x\in X_v$, $x=\mu f^\infty$ for
some $\mu\in \Lambda v$. Thus for any $(k,\ell),(k',\ell')\in \bbN^2$
we have $\sigma^{d(\mu)+(k,\ell)}(x)=\sigma^{d(\mu)+(k',\ell')}(x)$
and so $(x,(k-k', \ell-\ell'),x)\in \Iso(G_\Lambda)$ and so
$xG_\Lambda x\cong \bbZ^2$ and
$\Iso(G_\Lambda)^\circ|_{X_v} \cong \Lambda_v^\infty\times \bbZ^2$.
\end{proof}

\begin{lemma} \label{fibres}  Denote by 
  $\Lambda_V, \Lambda_W$
 the subgraphs of $\Lambda$ induced by the vertex sets $\{u,v\}, \{u, w\}$
respectively. Then
\begin{enumerate}
\item  $\Iso(G_{\Lambda_V})^\circ=\Iso(G_{\Lambda_V})^\circ|_{\{e^\infty\}}
  \cup \Iso(G_{\Lambda_V})^\circ|_{X_v}=
  \overline{\Iso(G_\Lambda)|_{X_v}}$;
\item $\Iso(G_{\Lambda_W})^\circ= \Lambda^\infty_W$.
\end{enumerate}
\end{lemma}

\begin{proof}
  Part~(a) follows from Lemma~\ref{iso Gv}.  For (b), let $x$ be any
  aperiodic path with $r(x)=w$ (e.g.,
  $x=(g_rg_b)(h_rh_b)(g_rg_b)^2(h_rh_b)(g_rg_b)^3 \cdots$). For any
  $\mu, \nu, \xi \in \Lambda_W$ with $s(\xi)=w$ and
  $r(\xi)=s(\mu)=s(\nu)$, $\mu \xi x = \nu \xi x$ if and only if
  $\mu=\nu$. Thus $Z_{G_W}(\mu, \nu) \subseteq \Iso(G_W)$ only if
  $\mu=\nu$.
\end{proof}

\begin{remark}For any open subgroupoid $H$ of an \'etale groupoid $G$,
  the inclusion $C_c(H)\hookrightarrow C_c(G)$ given by extension of
  functions by $0$, extends to an inclusion of
  $\cstarred(H)\hookrightarrow \cstarred(G)$ \cite[Lemma~2.7]{BFPR21}.
  We will often identify $\cstarred(H)$ with its image under this
  inclusion. In particular, we will identify
  $\cstarred(\Iso(G_{\Lambda})|_{X_v})$ with its image in
  $\cstarred(G_\Lambda)$ in Lemma~\ref{lambda ideal intersection}
  below.
\end{remark}

\begin{remark}
  For any \'etale groupoid (such as
  $G_\Lambda, G_{\Lambda_V}, G_{\Lambda_W}$), any open subgroupoid $H$
  is also \'etale.  Using \cite[Proposition~II.4.2]{ren80} we can
  consider any $a\in \cstarred(H)$ as a continuous function on $H$.
  We will often treat elements of $\cstarred(H)$ as continuous
  functions on $H$ in the sequel without further comment.
\end{remark}

\begin{lemma} \label{lambda ideal intersection}
The inclusion  $(\cstarred(\Iso(G_{\Lambda_V})^\circ,
\cstarred(\Iso(G_\Lambda)|_{X_v})$ has the ideal intersection
property.  
\end{lemma}

\begin{proof}
  For notational convenience, given $x\in \Lambda^\infty$,
  let \[S_x=\Iso(G_\Lambda)|_{\{x\}}.\] By Lemmas~\ref{iso Gv} and
  \ref{fibres}, $\Iso(G_{\Lambda_V})^\circ$ is an abelian group bundle
  over $\{e^\infty\}\cup X_v$, where the fibres over $X_v$ are
  isomorphic to $\bbZ^2$ and the fibre over $\{e^\infty\}$ is
  isomorphic to $\bbZ$.  By Lemma~\ref{dual 
    bundle}, \begin{equation*} \begin{split}  
      \cstarred(\Iso(G_{\Lambda_V})^\circ) = &C_0(\{(\xi,x): x\in  
      \Lambda^\infty_V, \ \xi \in \widehat{S_x}\}) \\ =  
      &C_0((\widehat{\mathbb{Z}^2}\times X_v) \cup  
      (\widehat{\mathbb{Z}} \times  
      \{e^\infty\})):\end{split} \end{equation*} the topology on
  $\{(\xi,x): x\in   \Lambda^\infty_V, \ \xi \in \widehat{S_x}\}$ 
   is given by $(\xi_i, x_i) \rightarrow (\xi, x)$ if and only if 
   \begin{itemize}   \item $x_i \rightarrow x$,    and  
     \item if
    $\gamma_i \rightarrow \gamma \in    
    \Iso(G_{\Lambda_V})^\circ$ with
    $\gamma_i \in     \Iso(G_{\Lambda_{V}})|_{\{x_i\}}$ and
    $\gamma \in     \Iso(G_{\Lambda_V})|_{\{x\}}$, then
    $\xi(\gamma_i) \rightarrow     \xi(\gamma)$.   
     \end{itemize}

Now suppose $I$ is an ideal in $\caop(\Iso(G_{\Lambda_V})^\circ)$
and $a\in I \setminus \{0\}$. We have \[\supp(a) \subseteq\{(\xi, x) \, : \, x \in \Lambda^\infty_V, \, \xi \in \widehat{S_x}\}.\]

We first show that there is an $\eta \in \widehat{\mathbb{Z}^2}$ and
$x \in X_v$ with $a(\eta, x) \neq 0$.  By way of contradiction suppose
not. Then since $a$ is nonzero, there is a
$\xi \in \hat{\mathbb{Z}}$ with $a(\xi, e^\infty) \neq 0$. Let
$(x_i)_{i \in \mathbb{N}}$ be a sequence in $X_v$ that converges to
$e^\infty$ and define $\xi_i=\xi \otimes 1$. We claim that
$(\xi_i, x_i) \rightarrow (\xi, e^\infty)$. Let
$(\gamma_i)_{i \in I}$ be a net converging to some $\gamma$ in
$\Iso(G_{\Lambda_V})^\circ$ where
$\gamma_i \in \Iso(G_{\Lambda_V})^\circ|_{\{x_i\}}$ and
$\gamma \in \Iso(G_{\Lambda_V})^\circ|_{\{e^\infty\}}$. By
Lemma~\ref{iso Gv}, there exists $m\in \bbZ$ such that
$\gamma=(e^\infty, (m,-m), e^\infty)$. Since $\gamma_i\to \gamma$ we
must have $\gamma_i \in Z(e_b^m, e_r^m)$ eventually, i.e.,
$\gamma_i = (x_i, (m,-m), x_i)$. By the definition of $\xi_i$, we
have $\xi_i(\gamma_i)=\xi(\gamma)$ for such $i$. By \ref{dual bundle}
and the continuity of $a$ we have $a(\xi_i,x_i)\neq 0 $ eventually, a
contradiction. 

Fix an $\eta \in \widehat{\mathbb{Z}^2}$ and $x \in X_v$ with $a(\eta, x) \neq 0$.  Let $\psi \in C_c(X_v)$ with $\psi(x)=1$. Then $\psi a \neq 0$ and $\psi a \in I \cap \cstar(\Iso(G_\Lambda)|_{X_v})$ giving the result. \qedhere
\end{proof}

\begin{lemma}\label{lambda cartan}
Define \[A_1=\cstarred(G_{\Lambda_V})\oplus \cstarred(G_{\Lambda_W})\quad \text{ and } \quad D_1=\cstarred(\Iso(G_{\Lambda_V})^\circ)\oplus \cstarred(\Iso(G_{\Lambda_W})^\circ).\]
Then $(A_1, D_1)$ is Cartan, where the conditional expectation $\Delta: A_1\to D_1 $ is induced by restriction of functions.
\end{lemma}
\begin{proof}
To show that the inclusion $\cstarred(\Iso(G_{\Lambda_V})^\circ)\oplus \cstarred(\Iso(G_{\Lambda_W})^\circ)\subseteq \cstarred(G_{\Lambda_V})\oplus \cstarred(G_{\Lambda_W})$ is Cartan we show each direct summand is Cartan. 

Note that $G_{\Lambda_W}$ is topologically principal as $\Lambda_W $ is aperiodic, and thus $C_0(\Lambda_W^\infty)$ is Cartan in $\cstarred(G_{\Lambda_W})$.  Recall from Lemma~\ref{fibres} that $\Iso(G_{\Lambda_V})^\circ$ is closed, so by \cite[Corollary 4.5]{BrownNagyReznikoffSimsWilliamsCaSuC*AlHaEtGr}  $\cstarred(\Iso(G_{\Lambda_V})^\circ)$ is Cartan in $\cstarred(G_{\Lambda_V})$. In both cases the conditional expectation is given by restriction of functions. Thus we get that the direct sum $\cstarred(\Iso(G_{\Lambda_V})^\circ)\oplus \cstarred(\Iso(G_{\Lambda_W})^\circ)\subseteq \cstarred(G_{\Lambda_V})\oplus \cstarred(G_{\Lambda_W})$ is Cartan.
\end{proof}

\begin{remark}\label{rmk: quotients}
For $F\in \{V,W\}$ and $\gamma$ in $G_\Lambda$, we have $s_{G_\Lambda}(\gamma)\in \Lambda_F^\infty$ if and only if $r_{G_\Lambda}(\gamma)\in \Lambda^\infty_F$.  Moreover, as $G_\Lambda$ is amenable, we have
\[
\cstarred(G_{\Lambda_F})\cong \cstarred(G_\Lambda)/\cstarred(G_{\Lambda}|_{\Lambda^\infty-\Lambda^\infty_F})
\]
(see, for example, \cite[Theorem~5.1]{WilliamsToKiGrC*Al}).
We define $q_F: \cstarred(G_\Lambda)\to \cstarred(G_{\Lambda_F})$ to be the resulting quotient map.
\end{remark}
\begin{proposition}\label{lambda envelope}
  Let \[A_1=\cstarred(G_{\Lambda_V})\oplus \cstarred(G_{\Lambda_W})\quad \text{ and } \quad D_1=\cstarred(\Iso(G_{\Lambda_V})^\circ)\oplus \cstarred(\Iso(G_{\Lambda_W})^\circ).\]
Then $(A_1, D_1)$ is the Cartan envelope for $(\cstarred(G_\Lambda), \cstarred(\Iso(G_\Lambda)^\circ))$.
\end{proposition}\begin{proof}
  To prove this proposition we use the uniqueness statement of the
  Cartan envelope \cite[Theorem~5.31]{Pit21}. Since we showed that
  $(A_1,D_1)$ is Cartan with conditional expectation $\Delta$, we must
  show:
\begin{itemize}
\item there is a regular $*$-monomorphism
  $$\alpha: (\cstarred(G_\Lambda), \cstarred(\Iso(G_\Lambda)^\circ))
  \to (A_1, D_1),$$
\item $\alpha(\cstarred(\Iso(G_\Lambda)^\circ))$ has the ideal
  intersection property in $D_1$, 
\item $\alpha(\cstarred(\Lambda))\cup D_1$ generates $A_1$, and 
\item $D_1 = \caop(\Delta(\alpha(\cstarred(G_\Lambda)))$.
\end{itemize}

Define
\[
\alpha:=q_V\oplus q_W: \cstarred(G_\Lambda) \to \cstarred(G_{\Lambda_V})\oplus \cstarred(G_{\Lambda_W})
\]
(see Remark~\ref{rmk: quotients}).  We show that
$\alpha$ is a regular $*$-monomorphism. That $\alpha$ is a
$*$-homomorphism follows from the definition. To see that $\alpha$ is
injective we compute its kernel. The kernel of $q_F$ is
$\cstarred(G_\Lambda|_{\Lambda^\infty-\Lambda_F^\infty})$ so that
$\ker(q_V)\cap\ker(q_W)=\{0\}$, hence $\alpha$ is injective.

To see that $\alpha$ is regular, it suffices to see that $q_F$ is
regular for $F\in \{V,W\}$. That $q_W$ is regular follows since
$\Iso(G_\Lambda)^\circ|_{\Lambda_W^\infty}=\Iso(G_{\Lambda_W})^\circ$.

We now show that $q_V$ is regular.  Since
$\cstarred(\Iso(G_{\Lambda_V})^\circ)$ is generated by the indicator
functions, $1_{Z_{G_{\Lambda_V}}(\mu,\nu)}$, for
$Z_{G_{\Lambda_V}}(\mu,\nu)\subseteq \Iso(G_{\Lambda_V})^\circ$, it
suffices to see that
\[q_V(n^*)1_{Z_{G_{\Lambda_V}(\mu,\nu)}}q_V(n)\in
\cstarred(\Iso(G_{\Lambda_V})^\circ)\]   for any $n\in
N(C^*_r(G_{\Lambda}) , C^*_r(\Iso(G_{\Lambda})^\circ)$ and 
$Z_{G_{\Lambda_V}}(\mu,\nu)\subseteq \Iso(G_{\Lambda_V})^\circ$.  For
this it suffices to show that
\[\supp
  \left(q_V(n^*)1_{Z_{G_{\Lambda_V}(\mu,\nu)}}q_V(n)\right)\subseteq\Iso(G_{\Lambda_V})^\circ.\]
Now if $e^\infty \notin r_{G_{\Lambda_V}}(Z_{G_{\Lambda_V}}(\mu,\nu))$
we have
$Z_{G_{\Lambda_V}}(\mu,\nu)=Z_{G_{\Lambda}}(\mu,\nu)\subseteq
\Iso(G_{\Lambda})^\circ$ and since $n$ is a normalizer we
have
\[q_V(n^*)1_{Z_{G_{\Lambda_V}}(\mu,\nu)}q_V(n)=n^*1_{Z_{G_{\Lambda}(\mu,\nu)}}n\in
  \cstarred(\Iso(G_{\Lambda})^\circ).\] Thus
$\supp (n^*1_{Z_{G_{\Lambda}}(\mu,\nu)}n)\in
\Iso(G_{\Lambda})^\circ\cap G|_{X_v} \subseteq
\Iso(G_{\Lambda_V})^\circ$ as desired.  So it remains to show that
$\supp
\left(q_V(n^*)1_{Z_{G_{\Lambda_V}}(\mu,\nu)}q_V(n)\right)\subseteq\Iso(G_{\Lambda_V})^\circ$
for $e^\infty \in r_{G_{\Lambda_V}}(Z_{G_{\Lambda_V}}(\mu,\nu))$.
 Since we are assuming
$e^\infty \in r_{G_{\Lambda_V}}(Z_{G_{\Lambda_V}}(\mu,\nu))\subseteq
\Iso(G_{\Lambda_V})^\circ$, by Lemma~\ref{iso Gv}, 
$\mu$ and $\nu$ must satisfy $d(\mu)-d(\nu)=(m,-m)$ for some $m \in \mathbb{Z}$.
\black
A
computation (details are provided in~\ref{ConvolutionComputation}
below) now shows that 
\begin{equation}\label{CCompA}
q_V(n^*)1_{Z_{G_{\Lambda_V}}(\mu, \nu)}q_V(n)(x,(s,t),y)=n^*n(x,(s-m,t+m),y).
\end{equation}
As $n$ is a normalizer, $n^*n\in \cstarred(\Iso(G_\Lambda)^\circ)$, so
that we must have
\[q_V(n^*)1_{Z_{G_{\Lambda_V}}(\mu, \nu)}q_V(n)(x,(s,t),y)\neq 0 \implies x=y.\] Moreover,
if $x=y=e^\infty$, $n^*n(e^\infty,(s-m,t+m),e^\infty)\neq 0$ only if
$(s-m, t+m)=(0,0)$ that is $(s,t)=(m, -m)$. Thus we get
$\supp q_V(n^*)1_{Z_{G_{\Lambda_V}}(\mu, \nu)}q_V(n^*)\subseteq \Iso(G_{\Lambda_V})^\circ$ and
thus
$\supp q_V(n^*)1_{Z_{G_{\Lambda_V}}(\mu, \nu)}q_V(n)\in \cstarred(\Iso(G_{\Lambda_V})^\circ)$; 
hence $q_V(n)$ is a normalizer. This completes the proof of regularity
of $\alpha$.

By Lemma~\ref{fibres},
\[\cstarred(\Iso(G_{\Lambda_W})^\circ)\cong C_0( \Lambda_W^\infty)\cong
\cstarred(\Iso(G_\Lambda)|_{X_w}).\]  Therefore,
$(D_1, \,\alpha(\cstarred(\Iso(G_\Lambda)^\circ )))$ has the ideal
intersection property if $\cstarred(\Iso(G_\Lambda|_{X_v}))$ has the
ideal intersection property in $\cstarred(\Iso(G_{\Lambda_V})^\circ)$;
this follows since $\Iso(G_\Lambda|_{X_v})$ is dense in
$(\Iso(G_{\Lambda_V}))^\circ$; see the proof of Lemma~\ref{lambda
  ideal intersection}.

That
$\alpha(\cstarred(G_{\Lambda}))\cup
\cstarred(\Iso(G_{\Lambda_V})^\circ)\oplus
\cstarred(\Iso(G_{\Lambda_W})^\circ)$ generates
$\cstarred(G_{\Lambda_V})\oplus \cstarred(G_{\Lambda_W})$ follows from
the fact that $q_F$ is surjective and
$\cstarred(G_{\Lambda_V})\oplus \cstarred(G_{\Lambda_W})$ contains
both $(1, 0)$ and $(0, 1)$.

To see that $D_1 = \caop(\Delta(\alpha(\cstarred(G_\Lambda))))$, it
suffices to show that $a\oplus 0$ and $0\oplus b$ are in
$\caop(\Delta(\alpha(\cstarred(G_\Lambda))))$ for any
$a\in \cstarred(\Iso(G_{\Lambda_V})^\circ)$ and
$b\in
C_0(\Lambda^\infty_W)=\cstarred(\Iso(G_{\Lambda_W})^\circ)$. Noting
that functions supported on sets of the form
$Z_{G_{\Lambda_V}}(\mu,\nu)$ generate
$\cstarred(\Iso(G_{\Lambda_V})^\circ)$, we may further assume that
$\supp b \subseteq U$, where $U=Z_{G_{\Lambda_V}}(\mu,\nu)$ for some
$\mu,\nu\in \Lambda_V$.

To begin, note that $q_V$ and $q_W$ are both surjective by definition
and thus there exist $c,d$ such that
$a\oplus c\in \Delta(\alpha(\cstar(G_\Lambda))) $ and
$d\oplus b\in \Delta(\alpha(\cstar(G_\Lambda)))$.

We will show
$1_{U} \oplus 0 \in \Delta(\alpha(\cstarred(G_\Lambda)))$, where $1_U$
is the characteristic function of $Z_{G_V}(\mu, \nu)$.  Since
$\Iso(G_{\Lambda_V})^\circ = \overline{\Iso(G_\Lambda)|_{X_v}}$ is
closed in $G_\Lambda$, we can use Tietze's extension theorem to extend
$1_{U}$ to a continuous compactly supported function $\omega$ on
$G_\Lambda$ such that $\omega|_{G_{\Lambda_V}}=1_{U}$. By subtracting
off $E(\omega)$ where $E: \cstar(G_\Lambda)\to C_0(G_\Lambda^{(0)})$
is the usual conditional expectation generated by restriction of
functions we can assume that
$\supp(\omega)\cap G_\Lambda^{(0)}=\emptyset.$ Thus
$\Delta(\phi_W(\omega))=0$, and so
$\Delta(\alpha(\omega))=1_{U} \oplus 0$ as desired.  Thus
$1_{r(U)}\oplus 0= (1_{U}\oplus 0)(1_{U}\oplus 0)^*$ is in
$\cstar(\Delta(\alpha(\cstarred(G_\Lambda))))$ and therefore
$a\oplus 0= (1_{r(U)}\oplus 0)(a\oplus c)$ is in
$\caop(\Delta(\alpha(\cstarred(G_\Lambda))))$ as well.

Since
$\cstarred(\Iso(G_{\Lambda_V})^\circ) \oplus 0 \subseteq
\cstarred(\Delta(\alpha(\cstarred(G_\Lambda)))$, in particular
$d\oplus 0 \in \caop(\Delta(\alpha(\cstarred(G_\Lambda)))$ and so
$0\oplus b = d\oplus b - d\oplus 0\in
\caop(\Delta(\alpha(\cstarred(G_\Lambda)))$ completing the claim.

So now \cite[Theorem~5.31]{Pit21} gives the result.
\end{proof}

For a groupoid $G$ and $\gamma\in G$ and $H$ an abelian subgroup of
$G|_{s(\gamma)}$, $\gamma H \gamma\inverse$ is an abelian subgroup of
$G|_{ r(\gamma)}$ isomorphic to $H$. For $\xi\in \hat H$ the map
$\gamma\cdot \xi: a\mapsto \xi (\gamma\inverse a\gamma)$ is an element
of $\widehat{\gamma H \gamma\inverse}$. This defines an action of $G$
on $\widehat{\Iso (G)^\circ}$.  As in the discussion on page 16 of
\cite{BFPR24}, an ideal of $C_r^*(\Iso (G)^\circ)$ is invariant in the
sense of Definition~\ref{def: N inv ideal} if and only if it is of the
form $C_0(U)$ for an open subset $\widehat{\Iso (G)^\circ}$ invariant
for this action.

We now compute the invariant regular open sets of the spectrum of $D_1$
in order to get the regular ideals of $\cstar(\Lambda)$.

\begin{proposition}
\begin{enumerate}
\item \label{lambda reg open 1}Regular open invariant subsets of the
  spectrum of $\cstarred(\Iso(G_{\Lambda_V})^\circ|_{X_v})$ are of the
  form $X_v\oplus U$ for some regular open subset $U$ of $\bbT^2$.
\item\label{lambda reg open 2} Regular open invariant subsets of the
  spectrum of $\cstarred(\Iso(G_{\Lambda_V})^\circ)$ are of the form
  $(X_v \oplus U)\cup ( \{e^\infty\}\oplus \pi_1(U)\pi_2(U)\inverse)$
  where $\pi_i:\bbT^2\to \bbT$ is the projection onto the $i$th factor
  and $U$ is some regular open subset of $\bbT^2$.
\item Regular open invariant subsets of $\widehat{D_1}$ are of the
  form
  $(X_v\oplus U)\cup ( \{e^\infty\}\oplus \pi_1(U)\pi_2(U)\inverse)$
  or
  $(X_v\oplus U)\cup ( \{e^\infty\}\oplus
  \pi_1(U)\pi_2(U)\inverse)\cup \Lambda^\infty_W$ where
  $\pi_i:\bbT^2\to \bbT$ is the projection onto the $i$th summand and
  $U$ is some regular open subset of $\bbT^2$.
\end{enumerate}
\end{proposition}

\begin{proof}
  For notational convenience given $x\in X_v$,
  put \[S_x=\Iso(G_{\Lambda_V})^\circ|_{\{x\}}.\] Now
  $S_x\cong \bbZ^2$.  So
  \[\widehat{S_x} \cong \widehat{\bbZ^2}\cong \bbT \times \bbT.\] So
  we identify $\widehat{S_x}$ with $\{x\} \oplus (\bbT\times
  \bbT)$. Under this identification
  \[(x,z_1\otimes z_2): (x,(m,n),x)\mapsto z_1^m z_2^n.\]

  Let $\Omega$ be the spectrum of
  $\cstarred(\Iso(G_{\Lambda_V})^\circ|_{X_v})$, then $\Omega$ is a
  bundle over $X_v$ whose fibres are given by $\widehat{S_x}$.  The
  action of $G_{\Lambda_V}$ on $\Omega$ is given by
  \begin{equation*} \begin{split} [(y,(p,q),x)\cdot(x,z_1\otimes z_2)](y,(m,n), y)&=(x,z_1\otimes z_2)((y,(p,q),x)\inverse (y,(m,n), y)(y,(p,q),x))\\
      &=(x,z_1\otimes z_2)(x,(-p,-q)+(m,n)+(p,q),x) \\ &=(y,z_1
      \otimes z_2)(y,(m,n),y).\end{split}\end{equation*} That is,
  $(y,(p,q),x)\cdot(x,z_1\otimes z_2)=(y,z_1\otimes z_2).$ Thus for an
  open set in $ \Omega$ to be invariant, it must be of the form
  $\Lambda_v^\infty\oplus U$ where $U$ is an open subset of
  $\bbT\times \bbT$.  Moreover since $X_v$ is discrete, $X_v \oplus U$
  is regular in $\Omega$ if and only if $U$ is regular in
  $\bbT\times \bbT$.

  For (b) we need to extend \eqref{lambda reg open 1} to $e^\infty$.
  By Lemma~\ref{lambda ideal intersection},
  $(e^\infty, (m,-m), e^\infty)\mapsto m$ defines an isomorphism
  $S_{e^\infty} \to \bbZ$.  We can characterize
  $\widehat{S_{e^\infty}} \cong \bbT$ by
  $z: (e^\infty, (m,-m), e^\infty)\mapsto z^m$. Now since
  $(e_b^n \alpha_b f^\infty, (m,n), e_b^n \alpha_b f^\infty)\to
  (e^\infty, (m,n), e^\infty)$ if and only if $n=-m$, we see that
  \[ (e_b^n \alpha_b f^\infty ,\zeta_n\otimes \omega_n) \to (e^\infty,
    z) \quad \text{iff} \quad \zeta_n\omega_n \inverse \to z \in
    \bbT.\] That is, the limit points of $\Lambda_v^\infty\oplus U$ in
  $\widehat{S_{e^\infty}}$ are the elements of
  $\overline{\pi_1(U)\pi_2(U)\inverse}$ where $\pi_i$ is the
  projection onto the $i$th component.  So the invariant regular open
  sets are characterized by
  $(\Lambda_v^\infty \oplus U)\cup \{e^\infty\}\oplus (
  \overline{\pi_1(U)\pi_2(U)\inverse})^\circ$, for $U$ a regular open
  set in $\bbT\times \bbT$.

  For the last statement, note that for $\mu\in \Lambda w$, the map
  $Z(w,w)\to Z(\mu,\mu)$ given by
  $(y,(0,0), y)\mapsto (\mu y, (0,0), \mu y)$ is a homeomorphism.
  Thus $G_{\Lambda_W}$ is minimal.  Moreover $G_{\Lambda_W}$ is
  aperiodic, thus $\Iso(G_{\Lambda_W})^\circ\cong \Lambda_W^\infty$
  and the only open invariant subsets of $\Lambda_W^\infty$ are
  $\emptyset$ and $\Lambda_W^\infty. $ Combining this with part
  \eqref{lambda reg open 2} gives the result.
\end{proof}

At this point we can use Theorem~\ref{thm: all latice isomorphisms} to
compute the regular ideals of $\cstarred(\Iso(G_\Lambda)^\circ)$,
$\cstarred(G_\Lambda)$ and $A_1$.  For example, the regular ideals of
$\cstarred(\Iso(G_\Lambda)^\circ)$ are the image under
$\alpha\inverse$ of regular ideals of $D_1$.  To compute this, note
that
\[\phi_V({\cstarred(\Iso(G_\Lambda)^\circ)})=\cstarred(\{e^\infty\}\cup
  \Iso(G_{\Lambda_V}|_{X_v})^\circ),\] so that \begin{eqnarray*}
  \alpha\inverse(C_0(X_v\oplus U)\cup \{e^\infty\} \oplus
  \overline{\pi_1(U)\pi_2(U)})\oplus \{0\})&=&C_0((X_v\oplus U)\cup
  \{e^\infty\}), \text{ and} \\\alpha\inverse(C_0((X_v\oplus U)\cup
  \{e^\infty\})\oplus C_0(\Lambda_W^\infty))&=&C_0((X_v\oplus U)\cup
  \Lambda_W^\infty). \end{eqnarray*}
\end{example}

\begin{dremark}{Computation} \label{ConvolutionComputation} For the
  interested reader, we conclude by giving the details of the
  computation establishing~\eqref{CCompA} above.  We assume the reader
  is familiar with the convolution formula which determines
  multiplication in the \cstar-algebra of an \'etale groupoid. Let
  $U=Z_{G_{\Lambda_V}}(\mu, \nu)  \subseteq \Iso(G_{\Lambda_V})^\circ$ with $d(\mu)-d(\nu)=(m,-m)$.  Then
 \begin{align}q_V(n^*)&1_{U}
    q_V(n)(x,(s,t),y)\notag \\
    &=\sum \left(n^*1_{U}\right)(x, (p,q), z) \, n\left((x, (p,q), z)^{-1}(x,(s,t),y)\right) \notag\\
&=\sum \sum n^*(x,(i,j),w)\, 1_{U}\left((x,(i,j),w)^{-1}(x, (p,q), z)\right) \, n(z,(s-p,t-q),y) \notag\\
&=\sum \sum n^*(x,(i,j),w)\, 1_{U}(w, (p-i,q-j), z)\,  n(z,(s-p,t-q),y). \notag\\\label{ConvCompB}
\intertext{Since $U   \subseteq \Iso(G_\Lambda)$ and $d(\mu)-d(\nu)=(m,-m)$, we have $(w, (p-i,q-j), z) \in U$ if and only if     $(w, (p-i,q-j), z)=(z, (m,-m), z)$.  This gives  $w=z$ and
  $(i,j)=(p-m, q+m)$. Hence continuing  the computation from the end
  of \eqref{ConvCompB}, we obtain}
&=\sum n^*(x,(p-m,q+m),z) \,  n(z,(s-p,t-q),y)\notag\\
&=\sum n^*(x,(p,q),z) \, n(z,(s-p-m,t-q+m),y),\notag\\
&=\sum n^*(x,(p,q),z) \, n\left((x, (p,q), z)^{-1}(x,(s-m,t+m),y)\right)
=n^*n(x,(s-m,t+m),y).\notag
\end{align}
\end{dremark}

\textit{Acknowledgement:} This work was supported by
 the American Institute of Mathematics SQuaREs Program.



\begin{thebibliography}{10}

\bibitem{Blackbook}
B.~Blackadar.
\newblock {\em Operator algebras}, volume 122 of {\em Encyclopaedia of
  Mathematical Sciences}.
\newblock Springer-Verlag, Berlin, 2006.
\newblock Theory of $C{^{*}}$-algebras and von Neumann algebras, Operator
  Algebras and Non-commutative Geometry, III.

\bibitem{Borys2020}
Clemens Borys.
\newblock The {F}urstenberg boundary of a groupoid.
\newblock arXiv:1904.10062, January 2020.

\bibitem{BCFS14}
Jonathan Brown, Lisa~Orloff Clark, Cynthia Farthing, and Aidan Sims.
\newblock Simplicity of algebras associated to \'{e}tale groupoids.
\newblock {\em Semigroup Forum}, 88(2):433--452, 2014.

\bibitem{BFPR21}
Jonathan~H. Brown, Adam~H. Fuller, David~R. Pitts, and Sarah~A. Reznikoff.
\newblock Graded {$C^*$}-algebras and twisted groupoid {$C^*$}-algebras.
\newblock {\em New York J. Math.}, 27:205--252, 2021.

\bibitem{BFPR22}
Jonathan~H. Brown, Adam~H. Fuller, David~R. Pitts, and Sarah~A. Reznikoff.
\newblock Regular ideals of graph algebras.
\newblock {\em Rocky Mountain J. Math.}, 52(1):43--48, 2022.

\bibitem{BFPR24}
Jonathan~H. Brown, Adam~H. Fuller, David~R. Pitts, and Sarah~A. Reznikoff.
\newblock Regular ideals, ideal intersections, and quotients.
\newblock {\em Integral Equations and Operator Theory}, 96(3):1--31, 2024.

\bibitem{BNR14}
Jonathan~H. Brown, Gabriel Nagy, and Sarah Reznikoff.
\newblock A generalized {C}untz-{K}rieger uniqueness theorem for higher-rank
  graphs.
\newblock {\em J. Funct. Anal.}, 266(4):2590--2609, 2014.

\bibitem{BrownNagyReznikoffSimsWilliamsCaSuC*AlHaEtGr}
Jonathan~H. Brown, Gabriel Nagy, Sarah Reznikoff, Aidan Sims, and Dana~P.
  Williams.
\newblock Cartan subalgebras in {$C^*$}-algebras of {H}ausdorff \'{e}tale
  groupoids.
\newblock {\em Integral Equations Operator Theory}, 85(1):109--126, 2016.

\bibitem{ChodaCoBeSuSuDiC*CrPr}
Hisashi Choda.
\newblock A correspondence between subgroups and subalgebras in a discrete
  {$C\sp{\ast} $}-crossed product.
\newblock {\em Math. Japon.}, 24(2):225--229, 1979/80.

\bibitem{DuwenigGillaspyNortonReznikoffWrightCaSuNoPrTwGr}
A.~Duwenig, E.~Gillaspy, R.~Norton, S.~Reznikoff, and S.~Wright.
\newblock Cartan subalgebras for non-principal twisted groupoid
  {$C^*$}-algebras.
\newblock {\em J. Funct. Anal.}, 279(6):108611, 40, 2020.

\bibitem{Exel23}
Ruy Exel.
\newblock Regular ideals under the ideal intersection property.
\newblock arXiv:2301.10073v1, January 2023.

\bibitem{GonRoy22}
Daniel Gon\c{c}alves and Danilo Royer.
\newblock A note on the regular ideals of {L}eavitt path algebras.
\newblock {\em J. Algebra Appl.}, 21(11):Paper No. 2250225, 11, 2022.

\bibitem{Ham79}
Masamichi Hamana.
\newblock Injective envelopes of {$C\sp{\ast} $}-algebras.
\newblock {\em J. Math. Soc. Japan}, 31(1):181--197, 1979.

\bibitem{Ham81}
Masamichi Hamana.
\newblock Regular embeddings of {$C\sp{\ast} $}-algebras in monotone complete
  {$C\sp{\ast} $}-algebras.
\newblock {\em J. Math. Soc. Japan}, 33(1):159--183, 1981.

\bibitem{Ham82}
Masamichi Hamana.
\newblock The centre of the regular monotone completion of a {$C^{\ast}
  $}-algebra.
\newblock {\em J. London Math. Soc. (2)}, 26(3):522--530, 1982.

\bibitem{HRSW}
Robert Hazlewood, Iain Raeburn, Aidan Sims, and Samuel B.~G. Webster.
\newblock Remarks on some fundamental results about higher-rank graphs and
  their {$C^*$}-algebras.
\newblock {\em Proc. Edinb. Math. Soc. (2)}, 56(2):575--597, 2013.

\bibitem{KalKen17}
Mehrdad Kalantar and Matthew Kennedy.
\newblock Boundaries of reduced {$C^*$}-algebras of discrete groups.
\newblock {\em J. Reine Angew. Math.}, 727:247--267, 2017.

\bibitem{KennedyGroupoid}
Matthew Kennedy, Se-Jin Kim, Xin Li, Sven Raum, and Dan Ursu.
\newblock The ideal intersection property for essential groupoid
  {C$^*$}-algebras.
\newblock arXiv:2107.03980, 2021.

\bibitem{KenSch19}
Matthew Kennedy and Christopher Schafhauser.
\newblock Noncommutative boundaries and the ideal structure of reduced crossed
  products.
\newblock {\em Duke Math. J.}, 168(17):3215--3260, 2019.

\bibitem{KumPas00}
Alex Kumjian and David Pask.
\newblock Higher rank graph {$C^\ast$}-algebras.
\newblock {\em New York J. Math.}, 6:1--20, 2000.

\bibitem{KumPasRaeRen97}
Alex Kumjian, David Pask, Iain Raeburn, and Jean Renault.
\newblock Graphs, groupoids, and {C}untz-{K}rieger algebras.
\newblock {\em J. Funct. Anal.}, 144(2):505--541, 1997.

\bibitem{Kum86}
Alexander Kumjian.
\newblock On {$C\sp \ast$}-diagonals.
\newblock {\em Canad. J. Math.}, 38(4):969--1008, 1986.

\bibitem{MuhlyRenaultWilliamsCoTrGrC*AlIII}
Paul~S. Muhly, Jean~N. Renault, and Dana~P. Williams.
\newblock Continuous-trace groupoid {$C^\ast$}-algebras. {III}.
\newblock {\em Trans. Amer. Math. Soc.}, 348(9):3621--3641, 1996.

\bibitem{MurphyBook}
Gerard~J. Murphy.
\newblock {\em {$C^*$}-algebras and operator theory}.
\newblock Academic Press, Inc., Boston, MA, 1990.

\bibitem{Pit17}
David~R. Pitts.
\newblock Structure for regular inclusions. {I}.
\newblock {\em J. Operator Theory}, 78(2):357--416, 2017.

\bibitem{Pit21}
David~R. Pitts.
\newblock Structure for regular inclusions. {II}: {C}artan envelopes,
  pseudo-expectations and twists.
\newblock {\em J. Funct. Anal.}, 281(1):108993, 2021.

\bibitem{PittsCoStReInII}
David~R. Pitts.
\newblock Corrigendum to ``{S}tructure for regular inclusions. {II}: {C}artan
  envelopes, pseudo-expectations and twists'' [{J}. {F}unct. {A}nal. 281 (1)
  (2021) 108993].
\newblock {\em J. Funct. Anal.}, 284(2):Paper No. 109748, 2023.

\bibitem{PittsIrMaIsBoReOpSeReId}
David~R. Pitts.
\newblock Irreducible maps and isomorphisms of {B}oolean algebras of regular
  open sets and regular ideals.
\newblock {\em Proc. Amer. Math. Soc.}, 153(6):2713--2727, 2025.

\bibitem{Pit25}
David~R. Pitts.
\newblock Pseudo-{C}artan inclusions.
\newblock arXiv:2502.01975v3, July 2025.

\bibitem{PitZar15}
David~R. Pitts and Vrej Zarikian.
\newblock Unique pseudo-expectations for {$C^*$}-inclusions.
\newblock {\em Illinois J. Math.}, 59(2):449--483, 2015.

\bibitem{ren80}
J.~Renault.
\newblock {\em A Groupoid Approach to {$C^*$}-Algebras}.
\newblock Number 793 in Lecture Notes in Mathematics. Springer-Verlag, New
  York, 1980.

\bibitem{Ren08}
Jean Renault.
\newblock Cartan subalgebras in {$C^*$}-algebras.
\newblock {\em Irish Math. Soc. Bull.}, (61):29--63, 2008.

\bibitem{SaiWriBook}
Kazuyuki Sait\^o and J.~D.~Maitland Wright.
\newblock {\em Monotone complete {$C^*$}-algebras and generic dynamics}.
\newblock Springer Monographs in Mathematics. Springer, London, 2015.

\bibitem{Sch22}
Tim Schenkel.
\newblock Regular ideals of locally-convex higher-rank graph algebras.
\newblock {\em New York J. Math.}, 28:1581--1595, 2022.

\bibitem{simsnotes}
Aidan Sims.
\newblock {H}ausdorff \' etale groupoids and their {$C^*$-algebras}.
\newblock In Francesc Perera, editor, {\em in {``Operator algebras and
  dynamics: groupoids, crossed products, and {R}okhlin dimension'' by A. Sims,
  G. Szab\'{o}, and D. Williams}}, Advanced Courses in Mathematics. CRM
  Barcelona, pages 58--120. Birkh\"{a}user/Springer, Cham, 2020.

\bibitem{Vas23}
Lia Va\v{s}.
\newblock Annihilator ideals of graph algebras.
\newblock {\em J. Algebraic Combin.}, 58(2):331--353, 2023.

\bibitem{WilliamsToKiGrC*Al}
Dana~P. Williams.
\newblock {\em A tool kit for groupoid {$C^*$}-algebras}, volume 241 of {\em
  Mathematical Surveys and Monographs}.
\newblock American Mathematical Society, Providence, RI, 2019.

\bibitem{Yan16}
Dilian Yang.
\newblock Cycline subalgebras of {$k$}-graph {$C^*$}-algebras.
\newblock {\em Proc. Amer. Math. Soc.}, 144(7):2959--2969, 2016.

\bibitem{YinHongShengSiPrClRaRoC*Al}
Hong~Sheng Yin.
\newblock A simple proof of the classification of rational rotation
  {$C^\ast$}-algebras.
\newblock {\em Proc. Amer. Math. Soc.}, 98(3):469--470, 1986.

\end{thebibliography}

\end{document}